\documentclass{amsart}

\usepackage{amstext,amssymb,amsmath,amsbsy,mathrsfs,mathtools,amsfonts,amscd,exscale}
\usepackage{graphics, epstopdf}
\usepackage{color}
\usepackage{amsmath, amssymb}
\usepackage{enumerate}
\usepackage{bigints}
\usepackage{verbatim}
\usepackage{cleveref}

\usepackage{graphicx}
\usepackage{epstopdf}
\usepackage{subcaption}
\usepackage{float}
\usepackage{placeins}

\usepackage{bm}

\newcommand{\ve}{\varepsilon}

\renewcommand{\sc}{\textsc}

\newcommand{\MA}{Monge-Amp\`{e}re }
\newcommand{\Th}{\mathcal{T}_h}

\newtheorem{Theorem}{Theorem}[section]
\newtheorem{Lemma}[Theorem]{Lemma}

\theoremstyle{definition}
\newtheorem{Definition}[Theorem]{Definition}

\newtheorem{remark}[Theorem]{Remark}

\numberwithin{equation}{section}

\def\l({\left(}
\def\r){\right)}
\def\uve{u_{\varepsilon}}

\def\Nhi{\mathcal{N}_h^0}
\def\Nhb{\mathcal{N}_h^b}
\def\sdd{\nabla^2_{\delta}} 
\def\sddp{\nabla^{2,+}_{\delta}} 
\def\sddm{\nabla^{2,-}_{\delta}} 
\def\sddb{\bm{\nabla}^2_{\delta}} 
\def\sddbp{\bm{\nabla}^{2,+}_{\delta}} 
\def\sddbm{\bm{\nabla}^{2,-}_{\delta}} 

\def\Vperp{\mathbb{S}^{\perp}}
\def\Vperpt{\mathbb{S}^{\perp}_{\theta}}
\def\St{\mathbb S_{\theta}}
\def\Vh{\mathbb{V}_h}

\def\Cka{C^{2+k,\alpha}}

\def\Wti{W^2_{\infty}}
\def\Wtio{W^2_{\infty}(\Omega)}

\def\Bhi{B_i}
\def\interp{\mathcal I_h}

\def\eps{{\tt eps}}

\def\bu{\mathbf{u}}
\def\bw{\mathbf{w}}
\def\bz{\mathbf{z}}
\def\bv{\mathbf{v}}
\def\bF{\mathbf{f}}
\def\bT{\mathbf{T}}
\def\bD{\mathbf{D}}
\def\bH{\mathbf{H}}
\def\bV{\mathbf{V}}

\definecolor{red}{rgb}{1,0,0}
\definecolor{blue}{rgb}{0,0,1}

\begin{document}

\title[Two-scale method for the Monge-Amp\`ere Equation]{Two-scale
  method for the Monge-Amp\`ere Equation: Convergence to the viscosity solution}

\author{R. H. Nochetto$^1$}
\address{Department of Mathematics and Institute for Physical Science
  and Technology, University of Maryland, College Park, Maryland 20742}
\email{rhn@math.umd.edu}
\thanks{$^1$ Partially supported by the NSF Grant DMS -1411808, the
Institut Henri Poincar\'e (Paris) and the Hausdorff Institute (Bonn).}

\author{D. Ntogkas$^2$}
\address{Department of Mathematics, University of Maryland, College Park, Maryland 20742}
\email{dimnt@math.umd.edu}
\thanks{$^2$ Partially supported by the  NSF Grant DMS -1411808 and
 the 2016-2017 Patrick and Marguerite Sung Fellowship of the
 University of Maryland.}

\author{W. Zhang$^3$}
\address{Department of Mathematics, Rutgers University, New Brunswick, New Jersey 08854}
\email{wujun@math.rutgers.edu}
\thanks{$^3$ Partially supported by the  NSF Grant DMS -1411808 and the
Brin Postdoctoral Fellowship of the University of Maryland.}

\begin{abstract}
We propose a two-scale finite element method for the Monge-Amp\`ere
equation with Dirichlet boundary condition
in dimension $d\ge2$ and prove that it converges to the viscosity
solution uniformly. The method is inspired by a finite difference
method of Froese and Oberman, but is defined on unstructured grids
and relies on two separate scales:
the first one is the mesh size $h$ and
the second one is a larger scale that controls appropriate directions
and substitutes the need of a wide stencil. The main tools for the analysis 
are a discrete comparison principle and discrete barrier functions that
control the behavior of the discrete solution, which is continuous
piecewise linear, both close to the boundary 
and in the interior of the domain. \end{abstract}

\maketitle


\vspace{0.2cm}


\section{Introduction}
We consider the Monge-Amp\`ere equation with Dirichlet boundary condition: 
\begin{equation} \label{E:MA}
\left\{
\begin{aligned}
\det{D^2u}&=f  & {\rm in} \ &\Omega \subset \mathbb {R}^d,
\\ u &=g & {\rm on} \ &\partial \Omega,
\end{aligned}
\right.
\end{equation}
where $f \geq 0$ is uniformly continuous, $\Omega$ is a uniformly
convex domain (not necessarily $W^2_\infty$)
and $g$ is a continuous function. We seek a
\textit{convex} solution $u$ of \eqref{E:MA}, which is critical for \eqref{E:MA} to be elliptic and 
have a unique viscosity solution \cite{Gut}. 

The Monge-Amp\`ere equation has a wide spectrum of applications in optimal mass transport problems, geometry, nonlinear elasticity and meteorology. These applications lead to an increasing interest in the investigation of efficient numerical methods. 
There exist several methods for the Monge-Amp\`ere equation. 
These include the early work by Oliker and Prussner \cite{OlPr} for
the space dimension $d=2$, the vanishing moment methods by Feng and
Neilan \cite{FeNe1,FeNe2}, the penalty method of Brenner, Gudi, Neilan
\cite{BrenNeil}, least squares and augmented Lagrangian methods by
Dean and Glowinski \cite{DeGl1, DeGl2, Gl}, the finite difference
method proposed recently by Benamou, Collino and Mirebeau
\cite{BeCoMi, Mir}, and a new semi-Lagrangian method by Feng and Jensen \cite{FeJe:16}.
Our work is mostly motivated by the wide-stencil scheme proposed by
Froese and Oberman, who proved convergence of the scheme \cite{FrOb1}.
Awanou \cite{Aw} proved a linear rate of convergence for classical
solutions for the wide-stencil method, when applied to a perturbed
Monge-Amp\`ere equation with an extra lower order term $\delta u$;
the parameter $\delta > 0$ is independent of the mesh and appears in
reciprocal form in the rate. In contrast, our analysis hinges on
the discrete comparison principle and two discrete barrier functions,
which are
instrumental in proving convergence to the viscosity solution of
\eqref{E:MA}. Moreover, our methodology further leads to
pointwise error estimates, which we derive in \cite{NNZ:17}.

\subsection{Our contribution}

Our method hinges on the following
formula for the determinant of the semi-positive Hessian $D^2w$ of a smooth
function $w$ as in \cite{FrOb1}:
\begin{equation}\label{E:Det}
 \det{D^2w}(x) = \min\limits_{(v_1,\ldots,v_d) \in \Vperp} \prod_{j=1}^d v_j^TD^2w(x) \ v_j ,
 \end{equation}
where $\Vperp$ is the set of all $d-$orthonormal bases in $\mathbb{R}^d$. The minimum is achieved for the eigenvectors of $D^2w$ and is equal to $\prod_{j=1}^d \lambda_j$, where $\lambda_j, j=1,\ldots, d$ are the corresponding eigenvalues.
To discretize \eqref{E:Det} we introduce two scales $h$ and
$\delta$. We discretize the domain $\Omega$ by a shape regular and
quasi-uniform mesh $\Th$ with spacing $h$, and
construct a space $\mathbb{V}_h$ of continuous piecewise linear
functions associated with the mesh $\Th$. The second scale $\delta$ is
the length of directions we use to approximate second directional
derivatives by central second order differences
\[
\sdd  w (x;v) := \frac{ w(x+\delta v) -2w(x) +w(x-\delta v) }{ \delta^2}
\quad
\text{and}
\quad
|v| = 1 ,
\] 
for any $w\in C^0(\overline{\Omega})$;
this formula will be appropriately modified close to $\partial\Omega$.
We denote by $\uve$ our discrete solution, where $\varepsilon = (h,\delta)$
represents the two scales, and define the discrete \MA operator to be
$$
T_\ve[\uve](x_i) := \min\limits_{(v_1,\ldots,v_d)\in \Vperp} \prod_{j=1}^d  
\sdd \uve (x_i;v_j) ,
$$
where $x_i$ is a generic node of $\Th$.
This leads to a clear separation of scales, which is a key theoretical advantage
over the original wide stencil method of \cite{FrOb1}. This also
yields continuous dependence of $\uve$ on data, which we further exploit in
\cite{NNZ:17}. In fact, such continuous dependence result, along with 
the discrete comparison principle and the use of some discrete barrier functions
give rise to rates of convergence in $L^{\infty}(\Omega)$ for viscosity
solutions of \eqref{E:MA} under some additional regularity
requirements \cite{NNZ:17}.
To make the two-scale method practical,
we resort to fast search techniques within \cite{WalkerPaper,WalkerWeb}
to locate points $x_i\pm\delta v_j$, which may not be nodes of $\Th$
in general.

The main tool in the current work is the discrete comparison principle that
enables us to control the behavior of $\uve$ and prove its uniform
convergence to the unique viscosity solution $u$ of \eqref{E:MA} as
$\delta\to0$ and $h\delta^{-1}\to0$.
It is important to realize, as already observed in
\cite{FeJe:16}, that such a convergence
is not an immediate consequence of the theory developed
by Barles and Souganidis \cite{BaSoug}. This theory assumes that the discrete 
operator is consistent up to the boundary and that the boundary conditions
are treated in the viscosity sense; our operator $T_\ve$ is only
consistent at distance $\delta$ from the boundary and our notion of
Dirichlet condition is classical. Moreover, the theory of \cite{BaSoug}
also hinges on a comparison principle 
for the underlying equation, which in the case of the \MA equation
\eqref{E:MA} requires 
that the subsolution and supersolution constructed through the limit 
supremum and limit infimum of $u_\ve$
be convex.

We present two proofs of uniform convergence.
The first one, discussed in Sections \ref{S:smooth-approx} and
\ref{S:convergence}, relies on regularization 
of data $f,g$ and $\Omega$ and the discrete comparison 
principle. This approach sets the stage for proving rates of
convergence for the 2-scale method, which we derive in \cite{NNZ:17}.
Regularization is a natural device used already for Monge-Amp\`ere
by De Philippis and Figalli \cite{DePhilippisFigalli:13} as a PDE tool and Awanou
for numerical purposes \cite{Awanou:15}.
The second approach is along the lines of Barles and Souganidis \cite{BaSoug},
uses techniques similar to those developed by Feng and Jensen
\cite{FeJe:16}, and circumvents the two main issues mentioned above.
Controlling the behavior of $\uve$ in a $\delta$-neighborhood of the
boundary $\partial\Omega$ is critical to both approaches. This is
achieved via a discrete barrier function discussed in Section
\ref{S:barriers}; similar constructions are discussed in
\cite{FeJe:16,NoZh,NoZh2}.

To showcase the performance of our 2-scale method,
we present computational experiments 
for a classical and a degenerate viscosity solution solved with a
semi-smooth Newton method. We obtain linear rates for both cases. 
We also present an example with unbounded $f$,
which does not fall within our theory, and still observe
convergence although with a reduced rate.

It is worth comparing the two-scale method with the Oliker-Prussner
method \cite{OlPr,NoZh2}. The former is easier to implement because it
does not require the explicit computation of subdifferentials, and is
formulated on shape regular meshes $\Th$ instead of cartesian meshes.
Although the coarse and fine scales $\delta$ and $h$ must only satisfy
$h\delta^{-1}\to0$ for convergence, rates of convergence require knowledge of
regularity of the exact solution $u$ of \eqref{E:MA} to
choose $\delta = \delta(h)$ \cite{NNZ:17} in contrast to
\cite{NoZh2}.

\subsection{Outline}

In \Cref{S:TwoSc} we introduce our method and the 
	main tool of our analysis,	the discrete comparison principle.
	In \Cref{S:Exist-Uniq}  we prove the existence and uniqueness
	 of our discrete solution. In \Cref{S:Consistency} we prove the 
	consistency of the discrete operator and in \Cref{S:Convergence} we
	prove the uniform convergence of the discrete solution to 
 the viscosity solution of \eqref{E:MA}.  Lastly, in
\Cref{S:NumEx} we document the performance of our method with
numerical experiments.

\section{ Two-Scale Method} \label{S:TwoSc}

\subsection{Ideal Two-Scale Method} \label{S:Ideal}
Let $\mathcal{T}_h$ be a shape-regular and quasi-uniform triangulation
with mesh size $h$. We denote by $\Omega_h$ the union of elements
of $\Th$ and we call it the computational domain.
Let $\mathcal{N}_h$ denote the nodes of
$\mathcal{T}_h$,
$\mathcal{N}_h^b := \{x_i \in \mathcal{N}_h: x_i \in \partial \Omega_h\}$
be the boundary nodes and
$\mathcal{N}_h^0 := \mathcal{N}_h \setminus \mathcal{N}_h^b$
be the interior nodes. We require that $\Nhb \subset \partial \Omega$, which in view of the convexity of $\Omega$ implies that $\Omega_h$ is also convex and $\Omega_h \subset \Omega$.
We denote by $\Vh$ the space of continuous
piecewise linear functions over $\mathcal{T}_h$.
We recall the notation $\Vperp$ for the collection of all
$d$-tuples of orthonormal bases and
$\mathbf{v} := (v_1,\ldots,v_d) \in \Vperp$ for a generic element, whence each component
$v_i$ belongs to the unit sphere $\mathbb{S}$ of $\mathbb{R}^d$.	
For $x_i\in\mathcal{N}_h^0$, we use the formula of centered second differences
\begin{equation} \label{E:2Sc2Dif}
\sdd w(x_i;v_j) :=   \frac{ w(x_i+ \rho \delta v_j) -2 w(x_i) + w(x_i- \rho \delta v_j) }{ \rho^2 \delta^2} ,
\end{equation}
where $0 < \rho \le 1$ is the biggest number such that both $x_i \pm
\rho \delta v_j \in \overline{\Omega}_h$ for all $v_j \in \mathbb{S}$;
we stress that $\rho$ need not be computed exactly.
This is well defined for any
$w \in C^0(\overline{\Omega})$, in particular for $w\in\Vh$.

We seek $\uve \in \Vh$ such that $u^{\varepsilon}(x_i)=g(x_i)$ for
$x_i \in \mathcal{N}_h^b$ and for $x_i \in \mathcal{N}_h^0$
\begin{equation} \label{E:2ScOp}
T_{\varepsilon}[\uve](x_i):=\min\limits_{\mathbf{v} \in \Vperp} \left( \prod_{j=1}^d \sddp \uve(x_i;v_j) - \sum_{j=1}^d \sddm \uve (x_i;v_j) \right) = f(x_i),
\end{equation}
where from now on we use the notation
$$
\sddp \uve(x_i;v_j) = \max{(\sdd \uve(x_i;v_j),0)}, \quad \sddm \uve(x_i;v_j) = -\min{(\sdd \uve(x_i;v_j),0)}.
$$
A similar definition was first proposed by Froese and Oberman in
\cite{FrOb1, Fr} for a finite difference method.
The key idea behind \eqref{E:2ScOp} is to enforce a suitable notion of
discrete convexity. To build intuition we explore this concept next.

\begin{Definition}[discrete convexity]\label{D:discrete-convexity}
  We say that $w_h \in \mathbb V_h$ is discretely convex  if
	$$
	\sdd w_h(x_i;v_j) \geq 0 \qquad \forall x_i \in \Nhi, \quad
        \forall v_j \in \mathbb S.
	$$
\end{Definition}
It is important to realize that this definition does not imply convexity in the usual
sense, which is rather tricky to achieve with piecewise polynomials
\cite{AguileraMorin09,NoZh2,Wachsmuth:16}. On the other hand, if
$w\in C^0(\overline{\Omega}_h)$ is convex, then its Lagrange interpolant
$\interp w$ satisfies $\interp w \ge w$, whence $\interp w$ is discretely
convex but not necessarily convex.

\begin{Lemma}[discrete convexity]\label{L:DisConv}	
If  $w_h \in \mathbb{V}_h$ satisfies
\begin{equation*} 
 T_{\varepsilon}[w_h](x_i) \geq 0 \quad \forall x_i \in \mathcal{N}_h^0,
\end{equation*}
then $w_h$ is \textit{discretely convex} and as a consequence
\begin{equation}\label{E:simpler-def}
T_{\varepsilon} [w_h](x_i)= \min_{\mathbf{v} \in \Vperp} \prod_{j=1}^d \sdd w_h(x_i;v_j),
\end{equation}
namely 
$$
\sddp w_h(x_i;v_j) = \sdd w_h(x_i;v_j),
\quad
\sddm w_h(x_i;v_j) =0
\quad\forall x_i\in \mathcal{N}_h^0,\quad\forall v_j\in\mathbb{S}.
$$
Conversely, if $w_h$ is discretely convex, then 
$ T_{\varepsilon}[w_h](x_i) \geq 0$ for all $x_i \in \mathcal{N}_h^0$.
\end{Lemma}
\begin{proof}
We distinguish two cases depending on whether $T_{\varepsilon} [w_h](x_i)>0$ or not.
Let $\mathbf{v}=(v_j)_{j=1}^d\in\Vperp$ be a $d$-tuple that realizes
the mininum in the definition of $T_{\varepsilon} [w_h](x_i)$ and note
that
\[
\prod_{j=1}^d \sddp w_h(x_i;v_j)\ge 0,
\quad
\sum _{j=1}^d \sddm w_h(x_i;v_j)\ge 0.
\]

{\it Case 1:} $T_{\varepsilon} [w_h](x_i)>0$.
If the difference of these two quantities is positive, then so must be
the first one. This implies that each factor $\sddp w_h(x_i;v_j) > 0$,
whence the second term must vanish. This readily yields \eqref{E:simpler-def}.

{\it Case 2:} $T_{\varepsilon} [w_h](x_i)=0$.
If instead the difference of the two quantities above vanishes, then
there are two possible situations. If the first quantity is strictly positive,
then the argument in Case 1 implies that the second quantity vanishes, which is
a contradiction. Therefore, the alternative option is that both
quantities vanish, whence
\[
\sddm w_h(x_i;v_j) =  0 \quad\forall j
\quad\Rightarrow\quad
\sdd w_h(x_i;v_j)\ge 0 \quad\forall j.
\]
This again implies that $w_h$ is discretely convex along with
\eqref{E:simpler-def}. Since the converse is trivial the proof is complete.
\end{proof}

\subsection{Practical Two-Scale Method} \label{S:Practical}
%
The ideal two-scale method of \Cref{S:Ideal} leads to the notion
of discrete convexity and Lemma \ref{L:DisConv} but cannot be implemented,
because the minimum in (\ref{E:2ScOp}) entails infinitely many options
for $\mathbf v \in \Vperp$. To render the two-scale method practical,
we introduce a finite discretization $\St \subset \mathbb S$ of the unit sphere governed by the parameter $\theta$: given $v \in \mathbb S$, there
exists $v^{\theta} \in \St$ such that \looseness=-1
$$
	|v-v^{\theta}| \leq \theta.
$$
Likewise, we approximate the set $\Vperp$ of $d$-orthonormal bases
by the finite set $\Vperpt \subset \Vperp$: for any
$\mathbf v^{\theta} = (v_j^{\theta})_{j=1}^d \in \Vperpt$,  $v_j^\theta \in \St$ and
there exists $\mathbf v = (v_j)_{j=1}^d \in \Vperp$
such that $|v_j - v_j^{\theta}| \leq \theta $ for all $1 \leq j \leq d$
and conversely.

If $\varepsilon := (h,\delta,\theta)$, the practical two-scale method
now reads: seek $\uve\in\Vh$ such that $\uve(x_i)=g(x_i)$ for $x_i \in \Nhb$ and for $x_i \in \Nhi$
\begin{equation} \label{E:PracticalOp}
T_{\varepsilon}[\uve](x_i):=\min\limits_{\mathbf{v} \in \Vperpt} \left( \prod_{j=1}^d \sddp \uve(x_i;v_j) - \sum_{j=1}^d \sddm \uve (x_i;v_j) \right) = f(x_i).
\end{equation}
We observe that
if we relax Definition \ref{D:discrete-convexity} (discrete convexity) to be
\[
\sdd w_h(x_i;v_j) \geq 0 \qquad \forall x_i \in \Nhi, \quad
\forall v_j \in \St,
\]
then \Cref{L:DisConv} (discrete convexity) is still valid
and we can take
\begin{equation} \label{E:PracicalOpShort}
T_{\varepsilon}[w_h](x_i)=\min\limits_{\mathbf{v} \in \Vperpt}  \prod_{j=1}^d \sdd w_h(x_i;v_j),
\end{equation}
provided $T_{\varepsilon}[w_h](x_i)\ge0$, as is the case of $\uve$ in
\eqref{E:PracticalOp}.

We now show that \eqref{E:PracticalOp} is monotone
and prove a comparison principle 
for $f \geq 0$.
\begin{Lemma}[monotonicity] \label{L:Monotonicity} 
	Let $u_h,w_h \in \Vh$ be discretely convex. If $u_h - w_h$ attains a 
	maximum at an interior node $z \in \Nhi$, then 
	$$
	T_{\varepsilon} [w_h](z) \geq T_{\varepsilon} [u_h](z).
	$$
\end{Lemma}
\begin{proof}
		If $u_h - w_h$ attains a maximum at $z\in \Nhi$, then 
	$$
	u_h(z)-w_h(z) \geq u_h(x_i) - w_h(x_i) \quad \forall x_i \in \mathcal N_h. 
	$$
	For suitably
	chosen $0<\rho\le 1$, the points $z \pm \delta \rho v_j\in\overline{\Omega}_h$ satisfy 
	$$
	u_h(z) - w_h(z) \geq u_h(z \pm \delta\rho v_j) - w_h(z \pm
	\delta\rho v_j)
	\quad\forall v_j\in \St,
	$$
	because this relation holds at the vertices of the simplices where
	$z \pm \delta \rho v_j$ belong to and both $u_h$ and $w_h$ are piecewise linear.
	Hence, \eqref{E:2Sc2Dif} implies
	\begin{equation*}
	\sdd u_h(z;v_j)\leq \sdd w_h(z;v_j)\ \ \forall v_j \in \St.
	\end{equation*}
	Since discrete convexity of $u_h$ and $w_h$ implies \eqref{E:PracicalOpShort}, 
        the proof is complete.
\end{proof}

\begin{Lemma}[discrete comparison principle] \label{L:DCP}
	Let $u_h,w_h \in \Vh$ with $u_h \leq w_h$ on the boundary $\partial \Omega_h$ be such that 
	\begin{equation*}
	T_{\varepsilon}[u_h](x_i) \geq T_{\varepsilon}[w_h](x_i)  \geq 0 \ \ \ \forall x_i \in \Nhi. 
	\end{equation*}
	Then, $u_h \leq w_h$ in $\Omega_h$.
\end{Lemma}

\begin{proof}
	Since $u_h,w_h \in \Vh$, it suffices to prove $u_h(x_i) \leq w_h(x_i)$
for all $x_i \in \Nhi$.
In view of \Cref{L:DisConv} (discrete convexity) and \eqref{E:PracicalOpShort},
we realize that both $u_h$ and $w_h$ are discretely convex and we
can rewrite the operator inequality as
follows:
\begin{equation*}
	\min_{\mathbf{v}  \in \Vperpt} \prod_{j=1}^d \sdd u_h(x_i;v_j)
        \geq \min_{\mathbf{v} \in \Vperpt} \prod_{j=1}^d \sdd
        w_h(x_i;v_j) \geq 0 \quad \forall x_i \in \Nhi.
\end{equation*}
The proof splits into two steps according to whether this inequality is strict or not.

\medskip
\textit{Step 1.}  We first consider the strict inequality
\begin{equation*}
\min_{\mathbf{v}  \in \Vperpt} \prod_{j=1}^d \sdd u_h(x_i;v_j) >
\min_{\mathbf{v} \in \Vperpt} \prod_{j=1}^d \sdd w_h(x_i;v_j)
 \ \ \ \forall x_i \in \Nhi. 
\end{equation*}
We assume by contradiction that there exists an interior node $x_k \in \Nhi$ such that
$$
u_h(x_k) - w_h(x_k) >0 
$$
and
$$
 u_h(x_k)-w_h(x_k) \geq u_h(x_i) - w_h(x_i) \quad \forall x_i \in \mathcal N_h. 
$$
Reasoning as in \Cref{L:Monotonicity} we obtain
$\sdd u_h(x_k;v_j)\leq \sdd w_h(x_k;v_j)$ for all $v_j \in \St$.
On the other hand, the original strict inequality at $x_i=x_k$ yields
\[ \min\limits_{\mathbf{v} \in \Vperpt} \prod_{j=1}^d \sdd
w_h(x_k;v_j) < \prod_{j=1}^d \sdd u_h(x_k;\overline{v}_j)
\]
for all possible directions $\overline{\bf  v}=(\overline{v}_j)_{j=1}^d\in\Vperpt$.
Choosing $\overline{\bf v}$ to be a $d$-tuple that realizes the
minimum of the left-hand side leads to a contradiction. 

\medskip
\textit{Step 2.}
We now deal with the non-strict inequality. We introduce the quadratic
strictly convex function $q(x) = \frac12\big(|x|^2-R^2\big)$, which satisfies
$q\le0$ on $\overline\Omega$ for $R>0$ sufficiently large and in particular $q\leq 0$ on $\partial \Omega_h$. Its Lagrange interpolant $q_h=\interp q$ is discretely convex and
\[
\sdd q_h(x_i;v_j) \geq \sdd q(x_i;v_j) = \partial^2_{v_jv_j} q(x_i) = 1
\quad\forall x_i \in \Nhi,
\quad\forall v_j\in \St.
\]
For arbitrary $\alpha>0$, the function $u_h+\alpha q_h$ satisfies
$u_h+\alpha q_h \le u_h\le w_h$ on $\partial\Omega_h$ and
\[
\sdd ( u_h + \alpha q_h)(x_i ; v_j) \geq \sdd u_h (x_i ;v_j)
+ \alpha > \sdd w_h (x_i ;v_j)
\quad\forall x_i \in \Nhi,
\quad\forall v_j\in \St,
\]
whence $T_{\varepsilon}[u_h+\alpha q_h](x_k) > T_{\varepsilon} [w_h](x_k)$.
Applying Step 1 we deduce
\[
u_h + \alpha q_h \le w_h \quad\forall \alpha>0.
\]
Taking the limit as $\alpha\to0$ gives the asserted inequality.
\end{proof}

\section{Existence and Uniqueness}\label{S:Exist-Uniq}
We now prove existence and uniqueness of a discrete solution 
$\uve \in \Vh$ of \eqref{E:PracticalOp}.

\begin{Lemma} [existence, uniqueness and stability]\label{L:Existence}
There exists a unique $\uve\in\Vh$ that solves the discrete Monge-Amp\`ere
equation \eqref{E:PracticalOp}. The solution $\uve$ is stable in the 
sense that $\|\uve\|_{L^{\infty}(\Omega)}$
does not depend on the parameters $\epsilon=(h,\delta,\theta)$ of the method.
\end{Lemma}

\begin{proof}
Since uniqueness is a trivial consequence of Lemma \ref{L:DCP}
(discrete comparison principle), we
just have to prove existence.
To this end, we construct a
monotone sequence of discrete convex functions $\left\{  u_h^k \right\}_{k=0}^\infty$,
starting with the initial iterate $u_h^0\in\Vh$ that satisfies
$u_h^0=\interp g$ on $\partial\Omega_h$ and
\[
T_{\varepsilon}[u_h^0](x_i) \geq f(x_i) \quad\forall x_i\in\Nhi.
\]

\textit{Step 1 : Existence of $u_h^0$.} 
We repeat the calculations of Step 2 in \Cref{L:DCP} (discrete comparison principle) for 
$$
q(x) = \frac{1}{2}\|f\|_{L^\infty(\Omega)}^{1/d}|x|^2
$$
to obtain that for $q_h=\interp q$ and for all $x_i \in \Nhi$
$$
T_\varepsilon[q_h](x_i) \geq \|f\|_{L^\infty(\Omega)} \geq f(x_i).
$$
We utilize the stability of $\interp q$ in $L^\infty(\Omega_h)$ to deduce
$$
\|q_h\|_{L^{\infty}(\Omega_h)} \leq C_R \ \|f\|_{L^\infty(\Omega)},
$$
where $C_R$ is a geometric constant that depends on the domain $\Omega$.

We next observe that the set of convex functions $w$ satisfying a
continuous Dirichlet boundary condition on a uniformly convex domain is
non-empty. The solution $w\in C^0(\overline{\Omega})$ of the
homogeneous Dirichlet problem \eqref{E:MA} is one such function \cite[Theorem 1.5.2]{Gut}. 
Let $w$ be convex and solve \eqref{E:MA} with $f=0$ and
Dirichlet condition $w = g - q$, whence $w_h := \interp w$ satisfies
\begin{equation*}
T_{\varepsilon}[w_h](x_i) \geq 0 \quad\forall x_i \in \Nhi.
\end{equation*}
and $w_h = \interp g - q_h$ on $\Nhb$.
We define the initial iterate to be
\[ u_h^0 := w_h + q_h \]
and note that $u_h^0$ is discretely convex and satisfies the
Dirichlet condition $u_h^0 = \interp g$ on $\partial\Omega_h$.
Since all the terms in
$T_{\varepsilon}[u_h^0](x_i)$ are non-negative, we also deduce
\[
T_{\varepsilon}[u_h^0](x_i) =  \min_{\mathbf{v}\in\Vperpt} \prod_{j=1}^d
\Big(\sdd w_h(x_i;v_j)+ \sdd q_h(x_i;v_j) \Big)
\geq f(x_i) \quad\forall x_i\in\Nhi.
\]
  
\textit{Step 2 : Perron Construction.}
We proceed by induction. Suppose that we have already a discretely
convex function $u_h^k\in\Vh$ that satisfies $u_h^k=\interp g$ on
$\partial\Omega_h$ and 
\begin{equation}\label{discrete-sub}
T_{\varepsilon}[u_h^k](x_i) \ge f(x_i) \quad\forall x_i\in\Nhi.
\end{equation}
We now construct $u_h^{k+1}\in\Vh$ such that $u_h^{k+1}\ge u_h^k$ in
$\Omega_h$, $u_h^{k+1} = \interp g$ on $\partial\Omega_h$ and
satisfies \eqref{discrete-sub}.
We consider all interior nodes in order and construct auxiliary
functions $u_h^{k,i-1}\in\Vh$ using the first $i-1$ nodes and starting
from $u_h^{k,0}:=u_h^k$ as follows. At $x_i\in\Nhi$ we check whether or not
$T^{\varepsilon}[u_h^{k,i-1}](x_i) > f(x_i)$. If so, we
increase the value of $u_h^{k,i-1}(x_i)$ and denote the resulting
function by $u_h^{k,i}$, until
\[
T_{\varepsilon} [u_h^{k,i}](x_i) = f(x_i).
\]
This is possible because the centered second differences
\eqref{E:2ScOp} are strictly decreasing with increasing central value
for all directions. Expression \eqref{E:2ScOp} also shows that
this process potentially increases the centered second differences at other nodes
$x_j \ne x_i$, whence
\[
T_{\varepsilon} [u_h^{k,i}](x_j) \ge T_{\varepsilon}[u_h^{k,i-1}](x_j) \ge f(x_j)
\quad\forall x_j\ne x_i.
\]
We repeat this process with the remaining nodes $x_j$ for $i<j\le N$,
and set $u_h^{k+1} := u_h^{k,N}$ to be the last intermediate
function. By construction, we clearly obtain
\[
T_{\varepsilon}[u_h^{k+1}](x_i) \ge f(x_i),
\quad
u_h^{k+1}(x_i)\ge u_h^k(x_i)
\quad\forall x_i\in\Nhi.
\]
Our construction preserves the boundary values
$u_h^{k+1}=\interp g$ on $\partial\Omega_h$ and enforces the relation
$u_h^{k+1}\ge u_h^k$ in $\Omega_h$ because both $u_h^{k+1}, u_h^k$ are
piecewise linear functions.
 
\textit{Step 3 : Bounds.}
If $b_h := \max_{x_i\in\Nhb} g(x_i)$, then we see that
$b_h\in\Vh$ and
\[
T_{\varepsilon}[b_h](x_i) = 0 \le f(x_i) \le
T_{\varepsilon}[u_h^k](x_i)
\quad\forall x_i\in\Nhi,
\quad\forall k\ge 0.
\]
We apply Lemma \ref{L:DCP} (discrete comparison principle)
to infer that $u_h^k\le b_h$ for all
$k\ge0$. On the other hand, since
$\|u_h^0\|_{L^\infty(\Omega_h)}$ is bounded uniformly in $h$ and
$u_h^0 \le u_h^k$, we deduce the uniform bound
\[
\| u_h^k\|_{L^\infty(\Omega)} \le \Lambda
\]
with $\Lambda>0$ independent of the discretization parameters
$h,\delta$ and $\theta$.

\textit{Step 4 : Convergence.}
The sequence $\{u_h^k(x_i)\}_{k=1}^\infty$ is monotone and bounded above for all
$x_i\in\Nhi$, and hence converges. The limit
\[
u_\varepsilon (x_i) = \lim_{k\to\infty} u_h^k(x_i)
\quad\forall x_i\in\Nhi
\]
defines $u_\varepsilon \in \Vh$ and satisfies $u_\varepsilon=\interp g$ on
$\partial\Omega_h$. It also satisfies the desired equality
\[
T_{\varepsilon}[\uve](x_i)=f(x_i)
\quad\forall  x_i \in \Nhi,
\]
since
$T_{\varepsilon}[\uve](x_i)=\lim_{k\to\infty}T_{\varepsilon}[u_h^k](x_i)\ge
f(x_i)$
and if the last inequality were strict, then Step 2 could be applied to
improve $\uve$. This shows existence of a discrete solution $\uve$ of
\eqref{E:2ScOp} as well as the uniform bound $\|\uve\|_{L^\infty(\Omega)} \le \Lambda$.
\end{proof}

\section{Consistency} \label{S:Consistency}
%
We now quantify the operator consistency error in terms of
H\"older regularity of $D^2 u$. 
We start with the definitions of $\delta$-interior region
\begin{equation}\label{Omega-delta}
\Omega_{h,\delta} = \left\{ T \in \mathcal{T}_h \ : \  {\rm dist}
(x,\partial \Omega_h) \geq \delta  \ \forall x \in T  \right\},
\end{equation}
and $\delta$-boundary region
$$
\omega_{h,\delta} = \Omega_h \setminus \Omega_{h,\delta}.
$$
Moreover, given a node $x_i \in \Nhi$ we denote by
\begin{equation}\label{E:Bi}
B_i := \cup \{\overline{T}: T\in\Th, \, \textrm{dist }(x_i,T) \le \hat\delta\}
\end{equation}
where $\hat\delta := \rho \delta$ with $0<\rho\le 1$ is the biggest number so that
$x_i\pm\hat\delta v_j\in\overline{\Omega}_h$ for all $v_j\in\St$.

\begin{Lemma} [consistency of $\sdd \mathcal I_hu $] \label{L:Consistency}
Let $u \in W^2_\infty(B_i)$, $\mathcal{I}_hu$ be its Lagrange
interpolant in $\Omega_h$, and $B_i$ be defined in \eqref{E:Bi}.
The following two estimates are then valid:
\begin{enumerate}[(i)]
\item 
For all $x_i  \in \Nhi$ and all  $v_j \in \St$, we have
$$
 \left|\sdd \mathcal{I}_hu (x_i;v_j) \right|\leq C  |u|_{W^2_{\infty}(B_i) } .
$$
\item
If in addition $u\in C^{2+k,\alpha}(B_i)$ for $k=0,1$ and $\alpha \in (0,1]$,
then for all $x_i \in \Nhi \cap \Omega_{h,\delta}$ and all $v_j \in \St$, 
we have
$$
 \left|\sdd \mathcal{I}_hu (x_i;v_j) - \frac{\partial^2 u}{\partial v_j^2}(x_i) \right| \leq C \left(  |u|_{C^{2+k,\alpha}(B_i)} \delta^{k+\alpha}+ |u|_{W^2_{\infty}(B_i)}  \frac{h^2}{\delta^2}  \right).
 $$
\end{enumerate}
In both cases $C$ stands for a constant independent of the two scales
$h$ and $\delta$, the parameter $\theta$ and $u$.

\end{Lemma}

\begin{proof}
We split the proof into three steps.

{\it Step 1.} Let $x_i\in\Nhi$ and $v_j \in\St$. Since
\[
u(x_i+ \hat\delta v_j) - u(x_i) = \hat\delta \int_0^1 \nabla
u(x_i+t\hat\delta v_j) \cdot v_j \,dt,
\]
definition \eqref{E:2Sc2Dif} yields
\begin{align*}
  \sdd u(x_i;v_j)= \hat\delta^{-1}
  \int_0^1 \left( \nabla u(x_i+t \hat\delta v_j) - \nabla u(x_i-t \hat\delta v_j) \right) \cdot v_j \,dt
\end{align*}
Adding and subtracting $ \nabla u(x_i) \cdot v_j $ inside the
integral, we similarly arrive at
\begin{align*}
\sdd u(x_i;v_j) = \int_0^1 \int_0^1 t \ \left( D^2 u(x_i + st \hat\delta
v_j)
+ D^2u(x_i-st \hat\delta v_j) \right) : v_j \otimes v_j \,ds \,dt,
\end{align*}
which implies
\[
\left|\sdd u (x_i;v_j) \right|\leq |u|_{W^2_{\infty}(B_i) } .
\]

\medskip
{\it Step 2.} Let $x_i\in\Omega_{h,\delta}$ and assume that $u\in C^{2,\alpha}(B_i)$.
We prove the estimate
\[
\left|\sdd u(x_i;v_j) - \frac{\partial^2 u}{\partial v_j^2}(x_i) \right| \leq C |u|_{C^{2,\alpha}(B_i)} \ \delta^{\alpha}.
\]
Write $\sdd u(x_i;v_j) = I_1 + I_2$, where
$$ 
I_1= 2 \int_0^1 \int_0^1 t D^2u(x_i) : v_j \otimes v_j \,ds \,dt
= \frac{\partial^2 u}{\partial v_j^2}(x_i)
$$
and
$$
I_2= \int_0^1 \int_0^1 t \left( D^2 u(x_i + st \ \delta v_j) - 2
D^2u(x_i) + D^2u(x_i-st \ \delta v_j) \right) : v_j \otimes v_j \,ds \,dt.
$$
The fact that $u \in C^{2, \alpha}(B_i)$ gives
$$
 |D^2 u(x_i \pm st \delta v_j)- D^2u(x_i)| \leq C
 |u|_{C^{2,\alpha}(B_i)} \delta^{\alpha}, 
 $$
whence
$$
I_2 \leq C |u|_{C^{2,\alpha}(B_i)} \delta^{\alpha}.
$$
Combining $I_1$ and $I_2$, we deduce the asserted estimate for
$u \in C^{2,\alpha}(B_i)$ and $k=0$.
For $u \in C^{3,\alpha}(B_i)$, we exploit the symmetry of $I_2$
to express the integrand in terms of differences of $D^3u$ at points
$x_i\pm stz \, \delta v_j$ for $0<z<1$ and thus deduce
$$
I_2 \leq C |u|_{C^{3,\alpha}(B_i)} \delta^{1+\alpha}.
$$
This implies the estimate for $k=1$
$$
\left|\sdd u(x_i;v_j)-  \frac{\partial^2 u}{\partial
  v_j^2}(x_i)\right| \leq C |u|_{C^{3,\alpha}(B_i)} \delta^{1+\alpha}.
$$

{\it Step 3.}  We now study the effect of interpolation, for which it
is known that \cite{BrenScott}
$$
\|u - \mathcal I_hu\|_{\infty} \leq C \ |u|_{W^2_{\infty}(B_i)} h^2.
$$
Therefore, applying definition (\ref{E:2Sc2Dif}), we deduce for
$x_i\in\Omega_{h,\delta}$
$$
|\sdd (u-\mathcal I_h u)(x_i;v_j)| \leq C |u|_{W^2_{\infty}(B_i)} \frac{h^2}{\delta^2}.
$$
This completes the proof of (ii) for $k=0,1$. Otherwise, $\delta$
must be replaced by $\hat\delta = \rho\delta \ge Ch$ with $C>0$
depending only on shape regularity. Therefore, we see that
$h^2\hat\delta^{-2}\le C$, 
which combined with Step 1 yields the estimate in (i).
\end{proof}

We now extend the consistency analysis to the practical
two-scale operator $T_\epsilon$.

\begin{Lemma}[{consistency of $T_\varepsilon [\mathcal{I}_h u]$}] \label{L:FullConsistency}
Let $x_i \in \Nhi \cap \Omega_{h,\delta}$ and $B_i$ be defined as in
\eqref{E:Bi}. If $u \in C^{2+k,\alpha}(B_i)$  is convex with $0<\alpha\leq 1$ and
$k=0,1$, and $\mathcal I_h u $ is its piecewise linear interpolant, then 
\begin{equation}\label{E:FullConsistency}
	\left|  \det D^2u(x_i) - T_{\varepsilon}[\mathcal I_h u] (x_i)  \right| \leq C_1(d,\Omega,u) \delta^{k+\alpha} + C_2(d,\Omega,u) \left( \frac{h^2}{\delta^2} + \theta^2 \right),
\end{equation}
where
\begin{equation*}
	C_1(d,\Omega,u)= C |u|_{C^{2+k,\alpha}(B_i)}
        |u|_{W^2_{\infty}(B_i)}^{d-1},
        \quad
        C_2(d,\Omega,u) = C |u|_{W^2_{\infty}(B_i)}^d.
\end{equation*}
	If $x_i\in\Nhi$ and $u \in W^2_{\infty}(B_i)$, then (\ref{E:FullConsistency}) remains valid with $\alpha=k=0$ and $\Cka(\Bhi)$ replaced by $\Wti(\Bhi)$.
\end{Lemma}
\begin{proof}
We recall that $\interp u$ is discretely convex, namely $\sdd \interp u (x_i,v_j)\ge0$
for all $x_i\in\Nhi$ and $v_j\in\St$, because $u$ is convex.
Therefore, in view of \Cref{L:DisConv} (discrete convexity),
the definition of $T_{\varepsilon}[\interp u]$ reduces to
	\begin{equation*} 
	T_{\varepsilon}[\interp u](x_i) = \min_{\mathbf v \in \Vperpt} \prod_{j=1}^d \sdd \interp u(x_i;v_j).
	\end{equation*}

\textit{Step 1.} Let $\mathbf v = (v_j)_{j=1}^d \in \Vperpt$ be the
$d$-tuple that realizes the above minimum. Applying
(\ref{E:Det}) to the determinant of the Hessian of $u$, we see that
$$
	\det D^2 u(x_i) - T_{\varepsilon}[\interp u] (x_i) \leq \prod_{j=1}^d \frac{\partial ^2 u}{\partial v_j^2}(x_i) - \prod_{j=1}^d \sdd \interp u(x_i;v_j).
$$
We now invoke \Cref{L:Consistency} (ii)
(consistency of $\sdd \interp u)$ to write 
$$
	\left| \frac{\partial^2 u}{\partial v_j^2}(x_i) - \sdd \interp
        u(x_i;v_j)  \right| \leq C |u|_{\Cka(\Bhi)}
        \delta^{k+\alpha} + C |u|_{\Wti(\Bhi)}
        \frac{h^2}{\delta^2},
$$
where $k=0,1$. Given the multiplicative structure above,
utilizing  \Cref{L:Consistency} (i) we deduce 
$$
\det D^2u(x_i) - T_{\varepsilon}[\mathcal I_h u] (x_i)  \leq C_1(d,\Omega,u) \delta^{k+\alpha} + C_2(d,\Omega,u)  \frac{h^2}{\delta^2},
$$
where $C_1$ and $C_2$ are defined above.

\medskip
\textit{Step 2.} We now choose $\mathbf{v}=(v_j)_{j=1}^d \in \Vperp$ to be the
$d$-tuple that realizes the minimum in (\ref{E:Det}) for
$\det D^2u(x_i)$. We can then write
\[
T_{\varepsilon}[\interp u](x_i) - \det D^2u(x_i) \le 
I_1 + I_2
\]
where
\begin{align*}
I_1 = \prod_{j=1}^d \sdd \interp u(x_i;\hat{v}_j) - \prod_{j=1}^d \frac{\partial^2 u}{\partial \hat{v}_j^2}(x_i),
\qquad
I_2 =  \prod_{j=1}^d \frac{\partial^2 u}{\partial \hat{v}_j^2}(x_i)  - \prod_{j=1}^d \frac{\partial^2 u}{\partial v_j^2}(x_i),
\end{align*}
and
$\hat{\mathbf{v}}=(\hat{v}_j)_{j=1}^d \in \Vperpt $ is an
approximation of $\mathbf{v}$ satisfying $|v_j - \hat{v}_j| \leq
\theta$ for all $1 \leq j \leq d$. The first term $I_1$ obeys a
similar estimate to Step 1. For the second term $I_2$ we notice that $\hat{v}_j = v_j + w_j$ with $|w_j| \leq \theta$, whence 
$$
\frac{\partial^2 u}{\partial \hat{v}_j^2}(x_i) = \hat{v}_j^T D^2u(x_i) \hat{v}_j = \frac{\partial^2 u}{\partial v_j^2}(x_i) + 2 w_j^T D^2u(x_i)v_j + w_j^T D^2u(x_i)w_j.
$$
Using that $\hat{v}_j = v_j + w_j$, we observe that 
$$
1= |\hat{v}_j |^2 = |v_j|^2 + 2 v_j \cdot w_j + |w_j|^2
\qquad\Rightarrow\qquad
|v_j \cdot w_j | = \frac{1}{2} |w_j|^2 \leq \frac{1}{2}\theta^2.
$$
Since $D^2 u(x_i) v_j=\lambda_j v_j$, we thus obtain
$$
\left| \frac{\partial^2 u}{\partial \hat{v}_j^2}(x_i) - \frac{\partial^2 u}{\partial v_j^2}(x_i)    \right| \leq C \ \theta^2 \ |u|_{\Wti(\Bhi)}
$$
as well as
$$
I_2 \leq C \ \theta^2 \ |u|_{\Wti(\Bhi)}^d.
$$
This proves \eqref{E:FullConsistency}.

The remaining statement for $u \in W^2_{\infty}(B_i)$ is a
simple consequence of \Cref{L:Consistency} (i) and the above 2-step argument.
\end{proof}

\begin{remark}[regularity]
We give sufficient conditions for the regularity of $u$ in Lemma
\ref{L:FullConsistency}:
if $0 < f_0 \leq f(x) \leq f_1$ for all
$x\in\Omega$ and $f\in C^{\alpha}(\overline{\Omega}), g\in
C^3(\overline{\Omega})$, and $\partial\Omega\in C^3$, then
$u\in C^{2,\alpha}(\overline{\Omega})$ \cite[Theorem 1.1]{TrudingerWang:08}.
In such a
case, there exist $0<\lambda \le \Lambda <\infty$ depending on
$f,g$ and $\Omega$ such that \cite[Theorem 2.10]{DePhilippisFigalli:14}
$$
\lambda I \leq D^2u(x) \leq \Lambda I  \quad \forall x \in \Omega.
$$
Since $\left| \frac{\partial^2 u}{\partial v_j^2}(x_i) \right| \leq \Lambda$, 
the constants $C_1$ and $C_2$ in Lemma \ref{L:FullConsistency}
could also be written
$$
C_1(d,\Omega,u)= C \Lambda^{d-1}|u|_{C^{2+k,\alpha}(B_i)} ,
\quad C_2(d,\Omega,u) = C \Lambda^{d-1} |u|_{W^2_{\infty}(B_i)}.
$$
\end{remark}

\section{Convergence}\label{S:Convergence}
%
Lemma \ref{L:FullConsistency} (consistency of $T_\ve[\interp u]$) shows
interior consistency at distance $\delta$ to $\partial\Omega_h$ for
$u\in C^2(\overline{\Omega})$; hence the
Barles-Souganidis theory \cite{BaSoug} does not apply directly,
as stated in \cite{FeJe:16}. We compensate with the fact that
$\interp u-\uve$ vanishes on $\partial\Omega_h$ and cannot grow faster than
$C\delta$ at distance $\delta$ to $\partial\Omega_h$. We make this
statement rigorous via a barrier argument similar to those in
\cite{FeJe:16,NoZh,NoZh2}. To handle the behavior of $u-\uve$ inside $\Omega_h$
we utilize \Cref{L:DCP} (discrete comparison principle) and
\Cref{L:FullConsistency} (consistency of $T_\ve[\interp u]$). In both cases we need the
solution to be $C^2(\overline{\Omega})$, which may in general be false
for the viscosity solution and thus requires a regularization
argument involving data $(f,g,\Omega)$. We discuss these topics in this
section and give a variation of the Barles-Souganidis approach as well.

\subsection{Barrier Functions}\label{S:barriers}

We now introduce two discrete barrier functions, one to deal with the
boundary behavior and the other one to handle the interior behavior.

\begin{Lemma}[discrete boundary barrier] \label{L:Barrier} Let $\Omega$ be uniformly convex and  $E>0$ be arbitrary. For each node $z \in \Nhi$ with ${\rm dist}(z,\partial \Omega_h) \leq \delta$, there exists a function $p_h\in\Vh$ such that $T_{\varepsilon}[p_h](x_i) \geq E$ for all $x_i \in \Nhi$, $p_h \leq 0$ on $\partial \Omega_h$ and
\[
|p_h(z)| \leq CE^{1/d} \delta 
\]
with $C$ depending on $\Omega$.
\end{Lemma}

\begin{proof}
Take $z_1 \in \partial \Omega_h$ such that
$|z-z_1|= {\rm dist}(z, \partial \Omega_h) \leq \delta$.
Upon extending the segment joining $z$ and $z_1$, we find $z_2 \in \partial\Omega$
that satisfies the upper bound $|z_2-z_1| \leq C_1 h$ because
$\Omega$ is uniformly convex and thus Lipschitz but not necessarily $W^2_\infty$.
This implies that for $z_0 \in \partial \Omega$ such that $|z-z_0| =
\textrm{dist}(z,\partial \Omega)$, we have $|z-z_0| \leq |z-z_2| \leq
\delta + C_1 h \le C_2\delta$. We now make a change of coordinates so that $z_0$ becomes the origin and $z=(0,\ldots,0,|z-z_0|)$. Since $\Omega$ is 
uniformly convex, it  lies inside the ball
$$
x_1^2+x_2^2+ \ldots+x_{d-1}^2 + (x_d-R)^2 \leq R^2, 
$$
where the radius $R$ depends on $\partial \Omega$
which is not necessarily $W^2_\infty$. Under this coordinate system, let $p(x)$ be the quadratic polynomial
$$
p(x) = \frac{E^{1/d}}{2} \left( x_1^2+x_2^2+\ldots+x_{d-1}^2+(x_d-R)^2-R^2 \right) 
$$
and $p_h=\interp p$ be its piecewise linear Lagrange interpolant in $\Omega_h$.
We note that $p \leq 0$ on $\overline \Omega$ yields
$p_h \leq 0$ on $\partial \Omega_h$.
Since $p$ is convex and $\interp p \geq p$, we infer that
$$ 
T_{\varepsilon}[p_h](x_i) \geq T_{\varepsilon}[p](x_i) = E \quad \forall \ x_i \in \Nhi,
$$
where the last equality is a consequence of $p$ being quadratic and 
$$
\sdd p(x_i;v_j) = \partial^2_{v_jv_j} p(x_i)= E^{1/d} \quad \forall \ v_j \in \St .
$$
Moreover, since $|z-z_0| \leq C_2\delta$, we deduce
$|p_h(z)| \leq C_{\Omega} E^{1/d}\delta$, as asserted.
\end{proof}

The following barrier function $q_h$ and corresponding statement have
  already been used in the proof of Lemma \ref{L:DCP} (discrete
  comparison principle).

\begin{Lemma}[discrete interior barrier] \label{L:BarrierInterior}
Let $\Omega$ be contained in the ball $B(x_0,R)$ of center $x_0$ and radius
$R$. If $q(x):= \frac12\big( |x-x_0|^2 - R^2 \big)$, then its
interpolant $q_h:=\interp q\in\Vh$ satisfies
\[
T_\ve[q_h](x_i) \ge 1\quad\forall x_i\in\Nhi,
\qquad
q_h(x_i) \le 0 \quad\forall x_i\in\Nhb.
\]
\end{Lemma}
%

\subsection{Approximation by Smooth Problems}\label{S:smooth-approx}

For data $f,g$ uniformly continuous in $\Omega$, $f\ge0$,
and $\Omega$ uniformly convex, the regularity $u\in C^2(\overline{\Omega})$
which would yield small interior consistency error is not guaranteed.
We thus embark on a regularization procedure similar to that used by
DePhilippis-Figalli \cite{DePhilippisFigalli:13} and Awanou
\cite{Awanou:15}. We start with a result about continuous dependence
on data for viscosity solutions.

\begin{Lemma}[continuous dependence on data] \label{L:ContDepSmooth}
Given $f_1,f_2\in C(\overline{\Omega})$, $f_1,f_2\ge0$, and $g_1,g_2\in
C(\partial\Omega)$, let $u_1,u_2\in C(\overline{\Omega})$ be the
corresponding convex viscosity solutions of \eqref{E:MA}. Then there
exists a constant $C$ depending on $\Omega$ such that
\[
\|u_1 - u_2\|_{L^{\infty}(\Omega)} \le C
\|f_1-f_2\|_{L^\infty(\Omega)}^{1/d} + \|g_1-g_2\|_{L^\infty(\partial\Omega)}.
\]
\end{Lemma}
\begin{proof}
Let $q\le 0$ be the barrier function of Lemma \ref{L:BarrierInterior}
(discrete interior barrier) and
$F:=\|f_1-f_2\|_{L^\infty(\Omega)}^{1/d}$, $G:=\|g_1-g_2\|_{L^\infty(\partial\Omega)}$.
We consider the auxiliary function
\[
u_1^- := u_2 +  F q - G,
\]
which is a convex viscosity subsolution of \eqref{E:MA} with data
$(f_1,g_1)$. To prove this, let $\phi\in C^2(\Omega)$ and
$x_0\in\Omega$ be a point where $u_1^--\phi$ attains a maximum. This
implies that $u_2 - \big(\phi -  F q + G  \big)$ attains also a maximum
at $x_0$. Since $u_2$ is a viscosity subsolution of \eqref{E:MA},
and $D^2 q(x_0)=I$ is the identity matrix, we deduce
\[
\det \big(D^2\phi (x_0) - F I \big) \ge f_2(x_0) \ge 0.
\]
Formula \eqref{E:Det} for two symmetric positive
semi-definite matrices $A, B$ easily implies
\[
\det (A+B) \ge \det (A) + \det(B).
\]
Using this expression for $A = D^2\phi (x_0) - F I$ and
$B = F I$ we obtain
\[
\det (D^2\phi(x_0)) \ge f_2(x_0) + F^d = f_2(x_0) +
\|f_1-f_2\|_{L^\infty(\Omega)} \ge f_1(x_0).
\]
In addition, since $q \le 0$ in $\Omega$, the function $u_1^-$ satisfies
on $\partial\Omega$
\[
u_1^- \le u_2 - G = g_2 - \|g_1-g_2\|_{L^\infty(\partial\Omega)} \le g_1.
\]
These two properties of $u_1^-$ imply that $u_1^-$ is a viscosity
subsolution of \eqref{E:MA} with data $(f_1,g_1)$. Since $u_1^-$ is
also convex, the comparison principle for \eqref{E:MA} gives
\[
u_1^- \le u_1 \quad \Rightarrow \quad
u_2 - u_1 \le -F q + G.
\]
We similarly prove the reverse inequality, thus obtaining the
desired estimate.
\end{proof}

We stress the monotonicity estimate
\[
f_1 \ge f_2\ge 0, \quad g_1 \le g_2
\qquad\Rightarrow\qquad
u_1 \le u_2,
\]
which is a consequence of $u_1$ being a convex subsolution of \eqref{E:MA}
with data $(f_2,g_2)$.

Using the above result, we now show that we can approximate a viscosity 
solution $u$ of \eqref{E:MA} by regular (classical) solutions $u_n$.

\begin{Lemma}[approximation of viscosity solutions by smooth
    solutions] \label{L:Approximation}
Let $\Omega$ be uniformly convex, $f, g$ be uniformly
continuous in $\Omega$, $f \geq 0$,
and $u$ be the viscosity solution of \eqref{E:MA} with data $(f,g,\Omega)$.
Then, there exist a decreasing sequence of uniformly convex and smooth domains
$\Omega_n$ converging to $\Omega$ in the sense that the Hausdorff
distance $\textrm{dist}_H(\Omega_n,\Omega)\to0$,
a decreasing sequence of smooth functions $f_n >0$ such that $f_n \to f$
uniformly in $\Omega$, a sequence of smooth functions $g_n$ such that
$g_n \to g$ uniformly in $\Omega$, and a sequence of smooth classical
solutions $u_n$ of \eqref{E:MA} with data $(f_n,g_n,\Omega_n)$
such that $u_n\to u$ uniformly in $\Omega$ as $n\to\infty$.
\end{Lemma}
\begin{proof}
We prove the result in four steps.

\smallskip
\textit{Step 1: Domain Approximation.}
According to \cite{Blocki} there is a sequence of smooth and uniformly
convex domains $\widetilde{\Omega}_n\subset\Omega$ that increase to $\Omega$
in the sense that the Hausdorff distance
$\textrm{dist}_H(\widetilde{\Omega}_n,\Omega)\to0$. Since $\Omega$ is
convex, it is star-shaped with respect to any of its points. Let's
assume that the origin is contained in $\Omega$ and dilate the domains
$\widetilde{\Omega}_n$ so that the ensuing domains $\Omega_n$ satisfy:
$$
\Omega\subset\Omega_n\subset\Omega_m \quad m\le n;
\qquad
\textrm{dist}_H(\Omega_n,\Omega) \to 0 \quad n\to\infty.
$$
The domains $\Omega_n$ inherit the regularity of
$\widetilde{\Omega}_n$ as well as their uniform convexity. Given
$\delta_n\to0$ as $n\to\infty$, to be chosen later in Step 4, we
relabel $\Omega_n$ to be an approximate smooth domain so that
$\textrm{dist}_H (\Omega_n,\Omega)\le \delta_n$.

\medskip
\textit{Step 2: Data Regularization.}
Let $\widetilde{\Omega}$ be an auxiliary domain such that
$\Omega_n\subset\widetilde{\Omega}$ for all $n$. We now construct a
sequence $(f_n,g_n)$ of smooth functions defined in $\widetilde{\Omega}$
that converge uniformly in $\Omega$ to $(f,g)$. We first extend
$(f,g)$ to $\widetilde{\Omega}$ and let $\sigma(t)$ be the modulus of
continuity in $\widetilde{\Omega}$ for both $(f,g)$ \cite[Theorem 2.1.8.]{Eng}:
$$
|f(x)-f(y)|, |g(x)-g(y)| \le \sigma (|x-y|)
\quad
\forall x,y \in \widetilde{\Omega}.
$$
Let $\rho < \textrm{dist}_H(\Omega_n,\widetilde{\Omega})$ and let $\phi_{\rho} \geq 0$ 
be a standard smooth mollifier function with support in $B(0,\rho)$.
We have for $f_{\rho} = f * \phi_{\rho}$ that
$$
\begin{aligned}
|f_{\rho}(x) - f(x)| &=  \left| \int_{\widetilde\Omega} (f(x-y)-f(x))
\ \phi_{\rho}(y) \ dy  \right| \leq \sigma(\rho)
\quad\forall x\in\Omega_n
\end{aligned}
$$
because $\phi_{\rho}$ integrates to one. This implies that
$$
\begin{aligned}
\widetilde{f}_\rho(x) &:= f_{\rho}(x) + 2\sigma(\rho)
\geq f(x) - \sigma(\rho) + 2\sigma(\rho)
= f(x) +\sigma(\rho) > 0
\quad\forall x \in \Omega_n.
\end{aligned}
$$
We now take $\rho_1 \le \rho_2$ and observe that for all $x \in \Omega_n$
$$
\begin{aligned}
\widetilde{f}_{\rho_1}(x) &- \widetilde{f}_{\rho_2}(x)
= \big(f_{\rho_1}(x) +2\sigma(\rho_1)\big) - \big(f_{\rho_2}(x)+2\sigma(\rho_2)\big)
\\
& \leq f(x) +\sigma(\rho_1) +2\sigma(\rho_1) -f(x) +\sigma(\rho_2) -2\sigma(\rho_2)
=3\sigma(\rho_1) -\sigma(\rho_2) \le 0,
\end{aligned}
$$
if $\sigma(\rho_1) \leq \frac{\sigma(\rho_2)}{3}$.
We thus choose $\rho_n$ such that $\sigma_n = \sigma(\rho_n) = 4^{-n}$ and define
$f_n := \widetilde{f}_{\rho_n}$,
which is a strictly positive and decreasing 
sequence of functions satisfying the error estimate
\begin{equation}\label{error-f}
\sigma_n \le f_n(x) -f(x) \le 3 \sigma_n
\quad\forall \, x\in\Omega_n.
\end{equation}
Similarly, we regularize $g$ by convolution $g_\rho=g *\phi_\rho$ and
define $g_n := g_{\rho_n}$ to obtain \looseness=-1
\begin{equation}
\|g-g_n\|_{L^\infty(\Omega_n)} \le \sigma_n.
\end{equation}

\textit{Step 3: Boundary Behavior.}
Let $u_n$ be the smooth
classical solution of \eqref{E:MA} with data $(f_n,g_n,\Omega_n)$,
which satisfies $u_n\in C^{2,\alpha}(\Omega_n)$ with norms depending
on $n$ but uniform $\alpha$; this is possible because $(f_n,g_n,\Omega_n)$ are
smooth, $\Omega_n$ is uniformly convex, and $f_n>0$
\cite{CaNiSp} \cite[Theorem 1.1]{TrudingerWang:08}.

We now compare $g$ and $u_n$ at $z\in\partial\Omega$ without invoking any
regularity of $u_n$ but rather using a barrier argument. We start with
$g$: if $y\in\partial\Omega_n$ is the closest point to $z$, then
$|z-y|\le\delta_n$ and
$$
|g(z) - g(y)| \le \sigma(|z-y|) \le \sigma(\delta_n).
$$
On the other hand, we know that
$$
|g(y) - g_n(y)| \le \sigma(\rho_n).
$$
Let $p$ be the quadratic barrier function introduced in the proof of
\Cref{L:Barrier}, but now associated with $\Omega_n$ and $y\in\partial\Omega_n$.
We consider the (lower) barrier function
$$
b_n^-(x):= p(x) + g_n(y) + \nabla g_n(y) (x-y),
$$
which satisfies
$$
\det D^2 b_n^- = \det D^2 p \geq f_n \quad\textrm{in }\Omega_n
$$
for $E > \|f\|_{L^\infty(\widetilde{\Omega})}$
because $b_n^-$ is a linear correction of $p$. We assert that 
$b_n^- \leq g_n$ on $\partial\Omega_n$ provided $E$ also satisfies
$E \ge C \|g_n\|_{W^2_\infty(\Omega_n)}$ where $C$ depends on the
uniform convexity of $\Omega$. If this is true, then applying the
comparison principle \cite[Theorem 1.4.6]{Gut} to the smooth functions $b_n^-$ and $u_n$ with data $(f_n,g_n,\Omega_n)$ yields
$$
b_n^-(x) \le u_n(x) \quad\forall \, x\in \Omega_n.
$$
Taking $x=z$ and making use of the definition of $b_n^-$ results in
$$
-C E^{1/d} |z-y| + g_n(y) + \nabla g_n(y) (z-y) \le u_n(z),
$$
whence
$$
u_n(z) - g_n(y) \ge - C_n |y-z| \ge -C_n \delta_n.
$$
Similarly, upon letting $b_n^+(x) := -p(x) + g_n(y) + \nabla g_n(y)(x-y)$
be an upper barrier function, the preceding argument also shows
$$
u_n(z) - g_n(y) \le C_n |y-z| \le C_n\delta_n,
$$
whence the triangle inequality implies that for all $z\in\partial\Omega$
\begin{equation}\label{error-g}
|g(z) - u_n(z)| \le \sigma(\delta_n) + \sigma (\rho_n) + C_n \delta_n,
\end{equation}
where the constant $C_n$ depends on $g_n$ but is independent of $u_n$.
It remains to show
$$
b_n^-(x) \leq g_n(x) \quad\forall \, x \in \partial\Omega_n.
$$
We first observe that $b_n^-(y)=g_n(y)$ and the tangental gradients
$\nabla_{\partial\Omega} b_n^-(y)=\nabla_{\partial\Omega} g_n(y)$
by construction, but $g_n$ grows quadratically away from $y$
on $\partial\Omega_n$ whereas
$p$ is just negative on $\partial\Omega_n$. To quantify the last
statement, we let $y=0$ for simplicity
and resort to the uniform convexity of $\Omega$ (and thus
to that of every $\Omega_n$) to deduce the existence of two balls $B_R$
and $B_r$ tangent to $\Omega_n$ at $0\in\partial\Omega_n$ and so
that
$$
\Omega_n \subset B_r \subset B_R;
$$
hence $r<R$. Note that $0 \in \partial B_r, \partial B_R$ and the centers of these balls are $(0,\ldots,0,r)$ and $(0,\ldots,0,R)$, respectively. We denote $x'=(x_i)_{i=1}^{d-1}$ and note that
$x=(x',x_d)\in\partial B_r$ satisfy $|x'|^2+(x_d-r)^2 = r^2$, whence
$$
x_d \Big( 1 - \frac{x_d}{2r}  \Big) = \frac{1}{2r} |x'|^2
\quad\Rightarrow\quad
\frac{1}{2r} |x'|^2 \le x_d \le \frac{1}{r} |x'|^2,
$$
provided $x_d \le r$. This in turn implies for $1<\xi<\frac{R}{r}$ fixed
$$
p(x) \le p\Big(x',\frac{1}{2r} |x'|^2\Big) = \frac{E^{1/d}}{2}
\Big(1-\frac{R}{r} + \frac{1}{4r^2} |x'|^2 \Big)  |x'|^2 \le
\frac{E^{1/d}}{2} (1- \xi) |x'|^2 < 0,
$$
provided $|x'|^2 \le C_1:=4 r^2 \big(\frac{R}{r} - \xi \big)$ and $R$ is used in the definition of $p$.
Since $|x'|^2\leq r^2$ and $x_d^2 \leq \frac{|x'|^4}{r^2} \leq |x'|^2$, 
	we have that $|x|^2 = |x'|^2 + x_d^2 \le 2 |x'|^2$ and we deduce
$$
|x'|^2 \le C_1
\quad\Rightarrow\quad
p(x) \le - E^{1/d} \frac{(\xi-1)}{4} |x|^2 = -
E^{1/d} C_2 |x|^2 .
$$
On the other hand, for $x\in\partial B_r$ with $|x'| > C_1$
we infer that the distance from $x$ to $\partial B_R$ is strictly
positive whence
$$
p(x) \le -C_3 |x|^2.
$$
Since both constants $C_2,C_3$ depend only on $r,R$, we see that $p$
grows quadratically on $\partial B_r$ with a constant independent of
$n$, and thus on $\Omega_n\subset B_r$. To compare $b_n^-$ with $g_n$,
we recall that $g_n$ is a smooth function for Taylor formula
to give
$$
\Big|g_n(x) - g_n(0) - \nabla g_n(0) x\Big|
\le \frac{1}{2} |g_n|_{W^2_\infty(\Omega_n)}|x|^2
\qquad\forall \, x\in \Omega_n.
$$
We finally choose the factor $E$ in $b_n^-$ proportional to
$|g_n|_{W^2_\infty(\Omega_n)}$ and realize that $b_n^-(x)\le g_n(x)$
for all $x\in\partial\Omega_n$ as asserted.

\smallskip
\textit{Step 4: Uniform Convergence.}
We view both $u$ and $u_n$ as viscosity solutions of \eqref{E:MA}, the former
with data $(f,g,\Omega)$ and the latter with data $(f_n,u_n,\Omega)$.
Applying \Cref{L:ContDepSmooth} (continuous dependence on data),
along with \eqref{error-f} and \eqref{error-g},
we obtain
\begin{align*}
\|u_n-u\|_{L^{\infty}(\Omega)} &\leq C
\|f_n-f\|_{L^\infty(\Omega)}^{1/d} 
+ \|u_n-g\|_{L^\infty(\partial \Omega)} \\
& \le C  \sigma(\rho_n)^{1/d} + \sigma(\rho_n) + \sigma(\delta_n) + C_n \delta_n.
\end{align*}
Given an arbitrary number $\beta$ we first choose $\rho_n$ so that
$C\sigma(\rho_n)^{1/d} + \sigma_n(\rho_n) \le \frac{\beta}{2}$.
This choice determines the regularity of $g_n$, namely its
$W^2_\infty$ and $C^{2,\alpha}$ norms in $\widetilde{\Omega}$. Since
$C_n$ is proportional to $|g_n|_{W^2_\infty(\Omega_n)}$, we finally select
$\delta_n$ so that $\sigma(\delta_n) + C_n \delta_n \le\frac{\beta}{2}$.
This shows the desired uniform convergence of $u_n$ to $u$ in $\Omega$.
\end{proof}

\subsection{Uniform Convergence: Regularization Approach}\label{S:convergence}
  In this section we combine Lemma \ref{L:DCP} (discrete comparison principle),
  \Cref{L:FullConsistency} (consistency of $T_\ve[\interp u]$),
  \Cref{L:Barrier} (discrete boundary barrier),
  \Cref{L:BarrierInterior} (discrete interior barrier), and \Cref{L:Approximation}
  (approximation of viscosity solutions by smooth solutions)
  to prove uniform convergence of $\uve$ to $u$ in $\Omega$.

Since $\uve$ is defined in the computational domain $\Omega_h$,
and $\Omega_h\subset\Omega$, we extend $\uve$ to $\Omega$ as
follows. Given $x\in\Omega\setminus \Omega_h$ let
$z\in\partial\Omega_h$ be the closest point to $x$, which is unique
because $\Omega_h$ is convex, and let
\begin{equation}\label{extension}
  \uve (x) := \uve(z) = \interp g(z)
  \quad\forall \, x \in \Omega\setminus\Omega_h.
\end{equation}

\begin{Theorem}[uniform convergence] \label{T:Convergence}
Let $\Omega$ be uniformly convex, $f, g \in C(\overline{\Omega})$
and $f\ge0$ in $\Omega$.
The discrete solution $\uve$ of (\ref{E:2ScOp})
and (\ref{extension}) converges uniformly to the unique 
viscosity solution $u\in C(\overline{\Omega})$ of \eqref{E:MA}
as $\varepsilon = (h, \delta, \theta) \rightarrow 0$
and $\frac{h}{\delta} \rightarrow 0$.
\end{Theorem}

\begin{proof}
We first split
$$
		\|u-\uve\|_{L^\infty(\Omega)} \leq \|u-\uve\|_{L^\infty(\Omega_h)} + \|u-\uve\|_{L^\infty(\Omega \setminus \Omega_h)}
$$	
and then employ the triangle inequality to write
$$
\|u-\uve\|_{L^\infty(\Omega_h)} \le
\|u-u_n\|_{L^\infty(\Omega_h)} +
\|u_n- \interp u_n\|_{L^\infty(\Omega_h)} +
\|\interp u_n-\uve\|_{L^\infty(\Omega_h)}.
$$
Next, we recall that \Cref{L:Approximation} yields
$\|u-u_n\|_{L^\infty(\Omega_h)} \leq \|u-u_n\|_{L^\infty(\Omega)}\to0$ as $n\to\infty$. 
In addition, polynomial interpolation theory gives
$$
\|u_n- \interp u_n\|_{L^\infty(\Omega_h)} \le C
|u_n|_{W^2_\infty(\Omega)} h^2 \to 0,
$$
as $h\to0$ for $n$ fixed. On the other hand, \eqref{extension} yields
\[
|u(x)-\uve(x)| = |u(x)-\uve(z)| \leq |u(x)-u(z)| + |u(z) - \uve(z)|
\quad\forall \, x\in\Omega\setminus\Omega_h
\]
where $z\in\partial\Omega_h$.
If $\tau$ is the modulus of continuity of $u \in C(\overline{\Omega})$,
we have
$$
\|u-\uve\|_{L^\infty(\Omega\setminus\Omega_h)} \le 
\tau\big(\textrm{dist}_H(\Omega,\Omega_h)\big) + \|u-\uve\|_{L^\infty(\Omega_h)}
$$
Since $\textrm{dist}_H(\Omega,\Omega_h) \to 0$, as $h \to 0$,
the proof reduces to showing that $\|\interp u_n-\uve\|_{L^\infty(\Omega_h)}$
can be made arbitrarily small. We do this in three steps.

\smallskip
\textit{Step 1: Boundary Estimate.}
Let $p_h$ be the function of \Cref{L:Barrier}
(discrete boundary barrier) with constant $E_{n,1}:=C|u_n|_{W^2_\infty(\Omega)}^d+3\sigma_n$,
where $C|u_n|_{W^2_\infty(\Omega)}^d$ is the consistency error \eqref{E:FullConsistency}
from \Cref{L:FullConsistency} (consistency of $T_\ve[\interp u]$)
with $u_n$ in place of $u$ and
$3\sigma_n$ is a bound \eqref{error-f} for $\|f-f_n\|_{L^\infty(\Omega)}$. 
Since both $\uve$ and $p_h$ are discretely convex, we have
$$
 T_{\varepsilon}[\uve + p_h](x_i)  \geq T_{\varepsilon}[\uve](x_i)
    + T_{\varepsilon}[p_h](x_i) \geq f(x_i) + E_{n,1} \geq T_{\varepsilon}[\interp u_n](x_i)
$$
for all $x_i \in \Nhi$. Moreover, since \eqref{error-g} holds for
all $z \in \partial \Omega$ and $\Nhb\subset\partial\Omega$,
linear interpolation implies that $\interp u_n \ge \interp g - \xi_n = \uve - \xi_n$
on $\partial\Omega_h$ for all $h$,
where $\xi_n := \sigma(\rho_n) + \sigma(\delta_n) + C_n \delta_n$ and
$\delta_n \ge \textrm{dist}_H(\Omega_n,\Omega)$, whence
$\uve+p_h - \xi_n \le \interp u_n$ on $\partial\Omega_h$.
Consequently, for all
$z \in \Nhi$ such that ${\rm dist}(z,\partial \Omega) \leq 2\delta$,
\Cref{L:DCP} (discrete comparison principle) yields
\begin{equation*}\label{E:boundarybound}
  \uve(z) - C E_{n,1}^{1/d} \delta  - \xi_n \leq \interp u_n(z).
\end{equation*}
A similar argument with $\uve-p_h+ \xi_n$ gives rise to the
reverse estimate.

\smallskip
\textit{Step 2: Interior Estimate.}
We resort to the function $q_h$ of
\Cref{L:BarrierInterior} (discrete interior barrier) to construct a
discrete lower barrier $b_\ve^-$ as follows: let
$$
E_{n,2} := C|u_n|_{C^{2+\alpha}(\overline{\Omega})}
  |u_n|_{\Wtio}^{d-1} \delta^{\alpha} +
  C |u_n|_{\Wtio}^d \Big(\frac{h^2}{\delta^2}+\theta^2\Big)
  + 3\sigma_n
$$
and
$$
  b_{\varepsilon}^- := \uve 
  + E_{n,2}^{1/d} q_h
  	- C E_{n,1}^{1/d} \delta  - \xi_n .
$$
Since $q_h\le0$, Step 1 guarantees that
$b_{\varepsilon}^- \leq \interp u_n$ on $\partial \Omega_{h,\delta}$,
where $\Omega_{h,\delta}$ is defined in \eqref{Omega-delta}.
Applying \Cref{L:FullConsistency} (consistency of $T_{\varepsilon}[\interp u])$
for $u_n$ instead of $u$ implies
\begin{align*}
   T_{\varepsilon}[b_{\varepsilon}^-](x_i)   &\geq
   T_{\varepsilon}[\uve](x_i) + E_{n,2} = f(x_i) + E_{n,2} 
   \geq T_{\varepsilon}[\interp u_n] (x_i)
   \quad\forall \, x_i \in \Nhi\cap\Omega_{h,\delta}
\end{align*}
where we have used that both $\uve$ and $q_h$ are discretely convex
as well as \eqref{error-f}.
\Cref{L:DCP} (discrete comparison principle) yields
  $$
  b_\ve^- = \uve + E_{n,2}^{1/d} q_h
  - C E_{n,1}^{1/d} \delta  - \xi_n \leq \interp u_n
  \quad\textrm{in} \quad \Omega_{h,\delta}.
  $$
A similar argument with
$b_\ve^+ := \uve - E_{n,2}^{1/d} q_h + C E_{n,1}^{1/d} \delta + \xi_n$
results in $b_\ve^+ \ge \interp u_n$.

Combining these estimates with those of Step 1, we end up with
	\begin{equation} \label{E:Omegahbound}
	\|\uve-\interp u_n\|_{L^\infty(\Omega_h)} \le
	C E_{n,1}^{1/d} \delta + C E_{n,2}^{1/d} + \xi_n.
	\end{equation}
\smallskip
\textit{Step 3: Uniform convergence in $\Omega$.}
We finally proceed as in step 4 of the proof of Lemma \ref{L:Approximation}
(approximation of viscosity solutions by smooth solutions). Given
an arbitrary number $\beta>0$, we choose $\rho_n$ so that
$\sigma(\rho_n)\le\frac{\beta}{3}$. This dictates the regularity of
$g_n$ hidden in the constant $C_n$ of $\xi_n$, as well as that of $u_n$,
and allows us to select $\delta_n$
so that $\sigma(\delta_n)+C_n\delta_n\le\frac{\beta}{3}$; hence
$\xi_n\le\frac{2\beta}{3}$. We next take
$\delta, \frac{h}{\delta}$ and $\theta$ small enough, depending on
$u_n$, so that the first two terms of \eqref{E:Omegahbound}
are $\le\frac{\beta}{3}$ and thus $\|\uve - \interp u_n\|_{L^{\infty}(\Omega_h)}\leq \beta$.
This completes the proof.
\end{proof}

\subsection{Uniform Convergence: Barles-Souganidis Approach}\label{S:BarlesSouganidis}
In this section we adapt the approach of \cite{BaSoug} to our setting. 
Since \eqref{extension} extends the definition of discrete
solution $\uve$ to $\Omega$,
we let the limit supremum and limit infimum of $\uve$ be
$$
	u^*(x) = \limsup_{\varepsilon, \frac{h}{\delta} \to 0, z \to x} \uve(z), 
	\qquad
	u_*(x) = \liminf_{\varepsilon, \frac{h}{\delta} \to 0, z \to x} \uve(z),
$$
and observe that $u^*$ is upper semi-continuous and $u_*$ is lower
semi-continuous. We show that they attain the Dirichlet boundary
condition pointwise. Moreover, they are viscosity subsolution and
supersolution of \eqref{E:MA}, respectively.
An essential difficulty associated with
\eqref{E:MA}, already mentioned in \cite{FeJe:16}, is that viscosity
sub and supersolutions of \eqref{E:MA} must be convex for the
comparison principle to be applicable. Since $\uve$ is only
discretely convex, it is not obvious that $u^*$ and $u_*$ are convex.

To circumvent this issue we proceed as in \cite{FeJe:16}: we let
$\partial^{2,+}_{v_jv_j} u := \max\big( \partial^2_{v_jv_j} u,0\big)$,
$\partial^{2,-}_{v_jv_j} u := -\min\big( \partial^2_{v_jv_j} u,0\big)$,
introduce the continuous version of our ideal discrete operator in
\eqref{E:2ScOp}
\begin{equation*}
T[u] := \min_{\bv=(v_j)_{j=1}^d\in\Vperp} \left( \prod_{j=1}^d \partial^{2,+}_{v_jv_j} u
- \sum_{j=1}^d \partial^{2,-}_{v_jv_j}u \right),
\end{equation*}  
and show that $u$ is a convex viscosity solution of
\eqref{E:MA} if and only if $u$ is a viscosity solution of
the Dirichlet problem
\begin{equation} \label{E:MAFull}
  T[u] = f \quad\textrm{in }\Omega,
  \qquad
  u = g \quad\textrm{on }\partial\Omega,
\end{equation}
for which we do not require convexity because it is built-in the operator.
\begin{Lemma}[equivalence of viscosity solutions] \label{L:EquivalenceVisc}
If $f\in C(\Omega)$ satisfies $f \geq 0$, and $u\in C(\overline{\Omega})$,
then $u$ is a
viscosity solution of \eqref{E:MAFull} if and only if $u$ is a convex
viscosity solution of \eqref{E:MA}.
\end{Lemma}
\begin{proof}
Since $u$ is uniformly continuous in $\Omega$ the notion of
Dirichlet condition is classical in both cases. We thus verify the
equation in the viscosity sense.

\smallskip
\textit{Step 1: Necessity.} We rely on the notion of convexity
of a function $v\in C(\Omega)$ in the
viscosity sense: for test function $\phi\in C^2(\Omega)$ that touches
$v$ from above at a point $x\in\Omega$ the
smallest eigenvalue $\lambda_1[D^2\phi](x)$  of $D^2\phi$ at $x$ satisfies
$$
	\lambda_1[D^2\phi](x) \geq 0. 
$$
It is proven in \cite{Ob2} that a continuous function $v$ is convex if
and only if it is convex in the viscosity sense. We show that a
viscosity solution $u$ of \eqref{E:MAFull} is convex in the viscosity
sense and use this equivalence to deduce convexity of $u$.

We observe that $u$ being a viscosity solution of
\eqref{E:MAFull} implies that for $\phi\in C^2(\Omega)$ touching $u$ from above at
$x\in\Omega$, we have
$$
T[\phi](x) \geq f(x) \geq 0
$$
We argue as in \Cref{L:DisConv}: if there is a direction
$v_j\in\mathbb{S}$ for which $\frac{\partial^2\phi}{\partial v_j^2}(x)<0$,
then $T[\phi](x) < 0$ which contradicts the preceding statement. Therefore
$$
\frac{\partial^2\phi}{\partial v_j^2}(x) \geq 0
\quad\forall \, v_j \in \mathbb{S}
\qquad\Rightarrow\qquad
\lambda_1[D^2\phi](x)\ge0.
$$
This proves that $u$ is convex as well as
$$
\det D^2\phi(x) = T[\phi](x) = \min_{\bv=(v_j)_{j=1}^d\in\Vperp} \prod_{j=1}^d
\partial^2_{v_jv_j} \phi(x) \ge f(x)
$$
according to \eqref{E:Det}. This implies that $u$ is a convex
subsolution of \eqref{E:MA}.

To prove that $u$ is also a supersolution of \eqref{E:MA}, we recall
that the definition of viscosity solutions for \eqref{E:MA} uses
convex test functions $\phi\in C^2(\Omega)$ \cite{Gut}; hence
$\det D^2\phi=T[\phi]$. Consequently,
if $u-\phi$ attains a minimum at $x\in\Omega$, then
$$
\det D^2\phi(x) = T[\phi](x) \le f(x)
$$
whence $u$ is a supersolution of \eqref{E:MA}.

\smallskip
\textit{Step 2: Sufficiency.}
Let's assume now that $u$ is a convex viscosity solution of
\eqref{E:MA}, and $\phi\in C^2(\Omega)$ is a test function that touches $u$
at $x_0\in\Omega$. Inspired by \cite[Remark 1.3.3]{Gut}, we
decompose $\phi=q + r$ into a quadratic $q$ and a remainder $r$
$$ 
q(x) = \phi(x_0) + D\phi(x_0) (x-x_0) + \frac{1}{2} (x-x_0)^T
D^2\phi(x_0) (x-x_0),
\quad
r(x) = o(|x-x_0|^2);
$$
hence $D^2\phi(x_0)=D^2q(x_0)$.
If $q^\pm(x):=q(x) \pm \sigma |x-x_0|^2$, we then observe that
$q^+\ge \phi$ and $q^-\le\phi$ in a suitable neighborhood of $x_0$
provided $\sigma>0$. We take advantage of $q^\pm$ being quadratic
to realize that $q^\pm$ is convex if and only if $D^2q^\pm(x_0)\ge0$.

If $u-\phi$ attains a local max at $x_0$, so does
$u-q^+$ and $D^2\phi(x_0)\ge0$ because $u$ is
convex. Therefore, the quadratic $q^+$ is convex and must satisfy
\[
\det D^2 q^+(x_0) = \det (D^2 q(x_0) + 2\sigma I) \ge f(x_0),
\]
because $u$ is a viscosity solution of \eqref{E:MA}. Take the limit
$\sigma\downarrow 0$ to find out that $T[\phi](x_0)=\det D^2\phi(x_0)\ge f(x_0)$
whence $u$ is a viscosity subsolution of \eqref{E:MAFull}.

On the other hand, if $u-\phi$ attains a local min at $x_0$, so does
$u-q^-$. We have now two possible cases. If all the eigenvalues of
$D^2\phi(x_0)$ are strictly positive, then $q^-$ is a convex
quadratic for $\sigma$ sufficiently small. This in turn implies
\[
\det D^2 q^-(x_0) = \det (D^2 q(x_0) - 2\sigma I) \le f(x_0),
\]
as $u$ is a viscosity solution of \eqref{E:MA};
hence $T[\phi](x_0)=\det D^2\phi(x_0) \le f(x_0)$ upon letting $\sigma\downarrow 0$.
If any eigenvalue of $D^2\phi(x_0)$ is non-positive, then
$T[\phi](x_0)\le0$ by definition and $T[\phi](x_0) \le f(x_0)$ because
$f\ge0$. We thus deduce that $u$ is a viscosity supersolution of \eqref{E:MAFull},
whence a viscosity solution of \eqref{E:MAFull}, as asserted.
\end{proof}

	We are now ready to prove the convergence of our discrete
        solution $\uve$ to the viscosity solution $u$ of
        \eqref{E:MA}. 
        
\begin{Theorem}[uniform convergence] \label{T:ConvergenceBarSoug}
Let $\Omega$ be uniformly convex, $f\in C(\Omega) \cap L^\infty(\Omega)$
satisfy $f\ge 0$, and $g\in C(\partial\Omega)$.
The discrete solution $\uve$ of (\ref{E:2ScOp}) converges uniformly to the unique 
viscosity solution $u\in C(\overline{\Omega})$ of \eqref{E:MA}
as $\varepsilon = (h, \delta, \theta) \rightarrow 0$
and $\frac{h}{\delta} \rightarrow 0$.
\end{Theorem}
\begin{proof}
In view of \Cref{L:EquivalenceVisc} (equivalence of viscosity solutions),
we prove that $\uve$ converges to the viscosity solution of \eqref{E:MAFull}.
To this end, we have to deal with a test function $\phi\in C^2(\Omega)$
and its Lagrange interpolant $\phi_h = \interp \phi$. Without loss of
generality we may assume $\phi \in C^{2,\alpha}(\Omega)$.
We split the proof into five steps. 

\smallskip  
\textit{Step 1: Consistency.}
We have the following alternative to \eqref{E:FullConsistency}
$$
\big| T[\phi](x_0) - T_\ve[\phi_h](x_i)  \big|
\le C_1(\phi) \Big(\delta^\alpha + |x_0-x_i|^\alpha \Big) +
C_2(\phi) \Big( \frac{h^2}{\delta^2} + \theta^2 \Big),
$$
where the constants $C_1, C_2$ are defined in \Cref{L:FullConsistency}
(consistency of $T_\ve[\interp u](x_i)$) and depend on
$|\phi|_{C^{2,\alpha}(B_i)}$ and $|\phi|_{W^2_\infty(B_i)}$ with $B_i$
defined in \eqref{E:Bi}, and $x_0\in\Omega, x_i\in\Nhi\cap \Omega_{h,\delta}$.
The proof of this inequality
proceeds along the lines of those of Lemmas \ref{L:Consistency} and
\ref{L:FullConsistency}, except
that now we need to deal with the functions $s\mapsto \max(s,0)$ and
$s\mapsto\min(s,0)$ in the definitions of both $T$ and $T_\ve$
because $\phi$ may not be convex.
We exploit that these functions are Lipschitz with constant
$1$ to write
$$
\big| \nabla^{2,+}_\delta \phi_h(x_i;v_j) - \partial^{2,+}_{v_j v_j} \phi(x_0)  \big|
\lesssim |\phi|_{C^{2,\alpha}(B_i)} \Big(\delta^\alpha + |x_0-x_i|^\alpha \Big)
+ |u|_{W^2_\infty(B_i)} \frac{h^2}{\delta^2},
$$
together with a similar bound for the operators
$\nabla^{2,-}_\delta$ and $\partial^{2,-}_{v_j v_j}$.

\smallskip  
\textit{Step 2: Subsolutions.} We show that $u^*$ is a
viscosity subsolution of \eqref{E:MAFull};
likewise $u_*$ is a viscosity supersolution.
This hinges on monotonicity and consistency \cite{BaSoug}.
We must show that if $u^* - \phi$ attains a local maximum at $x_0\in\Omega$, we have
$$
T[\phi](x_0) \geq f(x_0);
$$
note that $u^*-\phi$ is upper semi-continuous and the local maximum
is well defined. 
Without loss of generality, we may assume that $u^* - \phi$ attains a
strict global maximum at $x_0\in\Omega$ \cite[Remark in p.31]{IL},
and $x_0\in \Omega_h$ for $h$ sufficiently small.
Let $\uve$ and $z_h$ be a sequence of functions and nodes such that
$$
		\lim_{\varepsilon, \frac{h}{\delta} \to 0, z_h \to x_0} \uve(z_h) = u^*(x_0).
$$
Let $x_h\in\mathcal{N}_h$ be a sequence of nodes so that $\uve - \phi_h$
attains a maximum at $x_h$.
We claim that $x_h \to x_0$ as $h \to 0$.
If not, then there exists a subsequence $x_h \to y_0$ such that $y_0 \neq x_0$.
Since $(\uve - \phi_h)(x_h) \geq (\uve - \phi_h)(z_h)$, passing to the limit we obtain 
$$
		(u^* - \phi) (y_0) \geq \limsup_{\varepsilon, \frac{h}{\delta} \to 0} (\uve - \phi_h)(x_h) \geq 
		\lim_{\varepsilon, \frac{h}{\delta} \to 0, z_h \to x_0} (\uve - \phi_h)(z_h) = 
		(u^* - \phi)(x_0).
$$
This contradicts the fact that $u^* - \phi$ attains a strict maximum at $x_0$.
Exploiting the fact that $\uve - \phi_h$ attains a maximum at $x_h$,
\Cref{L:Monotonicity} (monotonicity) yields
$$
		T_{\varepsilon}[\phi_h](x_h) \geq T_{\varepsilon}[\uve](x_h) = f(x_h).
$$
Since $f\in C(\Omega)$, to prove $T[\phi](x_0) \ge f(x_0)$                
we only need to show that as $\varepsilon, \frac h {\delta} \to 0$
$$
		T_{\varepsilon}[\phi_h](x_h) \to T[\phi](x_0).
$$
This is a consequence of Step 1 and the fact that $x_h\in\Omega_{h,\delta}$
for $\delta$ sufficiently small, because $x_0\in\Omega$, $x_h \to x_0$
and the sequence of $\Omega_h\uparrow\Omega$ is non-decreasing.

\smallskip
\textit{Step 3: Boundary Behavior.}
We now prove that $u^*=u_*=g$ on $\partial\Omega$ via a barrier
argument similar to those in \cite{FeJe:16,NoZh,NoZh2};
we proceed as in \cite{FeJe:16}. This is
essential to apply the comparison principle for operator $T$ to relate
$u_*, u^*$ and $u$ in Step 4.

Let $p_k$ be the quadratic function in the proof of
\Cref{L:Barrier} (discrete boundary barrier)
associated with an arbitrary boundary point $x\in\partial\Omega$
 (the origin in the construction of $p_k$) and with constant $E=k$. We
recall that $p_k(x) =0$ and $p_k(z) \leq 0$ for all $z \in \partial\Omega$
can be made arbitrarily large for $k\to\infty$ by virtue of the
uniform convexity of $\Omega$. A simple consequence is that the
sequence of points $x_k\in\partial\Omega$ where $g+p_k$
(resp. $g-p_k$) attains a maximum (resp. a minimum)
over $\partial\Omega$ converges to $x$.
  
We now observe that taking $w_h \equiv 0$ in \Cref{L:Monotonicity}
(monotonicity) implies the following maximum principle:
if a discretely convex function $u_h$ satisfies
$T_{\varepsilon}[u_h](x_i)>0$ for all $x_i\in\Nhi$,
then $u_h$ attains a maximum over $\overline{\Omega}_h$ on
$\Nhb \subset \partial \Omega$. 
Apply this to
$T_{\varepsilon}[\uve + \interp p_k] > 0$ to deduce that $\uve + \interp p_k$
attains its maximum on $\Nhb$. In view of \eqref{extension}, we
may assume $z\in\Omega_h$ in
$u^*(x) = \limsup_{\varepsilon, \frac{h}{\delta} \to 0, z \to x}\uve(z)$.
Consequently,
\begin{align*}
u^*(x) & \le \limsup_{\varepsilon, \frac{h}{\delta} \to 0, z \to x}
\big(\uve(z)+\interp p_k(z) \big) - \liminf_{\varepsilon, \frac{h}{\delta}
  \to 0, z \to x} \interp p_k(z) \\
& \le \limsup_{\varepsilon, \frac{h}{\delta} \to 0} \
\max_{z \in \Nhb } \ \big(g + p_k\big) (z) -  p_k(x) 
\le g(x_k) + p_k(x_k) \le g(x_k),
\end{align*}
because $\max_{\Nhb} g+p_k \le \max_{\partial \Omega } g + p_k$,
whence taking $k \to \infty$ yields  $u^*(x) \leq g(x)$.

On the other hand, since $T_\ve[\interp p_k](x_i) > T_\ve[\uve](x_i)$
for all $x_i\in\Nhi$ and $k$ big enough, Lemma \ref{L:Monotonicity} implies that
$\uve-\interp p_k$ attains a minimum on $\Nhb$. Therefore,
arguing as before
\begin{align*}
  u_*(x)\ge \liminf_{\varepsilon, \frac{h}{\delta} \to 0} \
  \min_{z \in \Nhb} \ \big(g - p_k\big)(z) +  p_k(x)
  \ge g(x_k) - p_k(x_k) \ge g(x_k),
\end{align*}
whence $u_*(x) \geq g(x)$. This in turn gives
$u^* \leq g \leq u_* \leq u^*$ on $\partial\Omega$
as asserted.

\smallskip
\textit{Step 4: Comparison.} To prove that $u^*=u_*$ in
$\overline{\Omega}$ we use the following comparison principle for
\eqref{E:MAFull}:
if $v^-$ is a subsolution and is upper semi-continuous in
$\overline{\Omega}$, $v^+$ is a supersolution and is lower semi-continuous in
$\overline{\Omega}$, and $v^- \le v^+$ on $\partial\Omega$, then
$v^- \le v^+$ on $\overline{\Omega}$.
This result falls under the umbrella of \cite[Theorem 3.3]{CraIshLi}.
It hinges on an argument mentioned in \cite[Section 5.C]{CraIshLi}
that is briefly described for a more general form of the \MA operator
in \cite[V.3]{IL}. Both operators in \eqref{E:MA} and 
\eqref{E:MAFull} satisfy the requirements posed in \cite{IL}.
We apply this comparison principle to $v^-=u^*$ and $v^+=u_*$, which
satisfy the assumptions in view of Steps 2 and 3, to obtain
$u^*\le u_*$ in $\overline{\Omega}$. Since $u^* \ge u_*$ by
definition, this results in $u^*=u_*$ in $\overline{\Omega}$.

\smallskip 
\textit{Step 5: Uniform Convergence.} Step 4 implies the pointwise limit
$$
u(x) = \lim_{\ve\to0, z\to x} \uve(z)
\quad\forall \, x\in\overline{\Omega}.
$$
To see that this gives rise to uniform convergence we argue by
contradiction. We assume that for every $\ve$ there exist a point
$x_\ve\in\overline{\Omega}$ such that
$
|u(x_\ve) - \uve(x_\ve)| \ge \sigma,
$
for some $\sigma>0$. Since $\overline{\Omega}$ is compact, there exists
a subsequence (not relabeled) $x_\ve\to
x_0\in\overline{\Omega}$. Computing the limit $\ve\to0$ in the last
inequality yields the contradiction
$
|u(x_0) - u(x_0)| \ge \sigma.
$
This concludes the proof.
\end{proof}

\section{Numerical Experiments} \label{S:NumEx}

We present three examples in the square domain $\Omega = \Omega_h =[0,1]^2$.
The fact that $\Omega$ is not uniformly convex does not affect the 
existence of our discrete solution $\uve$, as the Dirichlet datum $g$ is
the trace of a convex function; however this is beyond the assumptions 
of the convergence theory.  
We implement the 2-scale method within the
MATLAB software FELICITY \cite{WalkerWeb,WalkerPaper}. 
We first consider two examples with smooth Hessian and with discontinuous
Hessian, and observe linear experimental rates of convergence with
respect to $h$; we further investigate rates theoretically in \cite{NNZ:17}.
The third example entails an unbounded right hand side $f$ and is not
guaranteed to converge by theory. We still observe convergence
experimentally.
 
\subsection{Semi-Smooth Newton Method}

We solve the nonlinear algebraic equation \eqref{E:2ScOp} via a
damped semi-smooth Newton iteration. Let
$\bz := (z_h(x_i))_{i=1}^N \in\mathbb{R}^N$ stand for the vector of
nodal values of a generic $z_h\in\Vh$;
thus $N$ is the cardinality of $\mathcal{N}_h$.
If $\bu_n =\left ( \uve^n(x_i)  \right)_{i=1}^N$,
$\bD \bT_{\varepsilon}[\bu_n]$
is the Jacobian matrix of the nonlinear map $\bT_\varepsilon: \mathbb{R}^N\to\mathbb{R}^N$
at $\bu_n$, and $\bF=(f(x_i))_{i=1}^N$, then a Newton increment is given by
\[
\bD \bT_{\varepsilon}[\bu_n]  \ \bw_n = \bF - \bT_{\varepsilon}[\bu_n]
\]
and the $n$-th Newton step by
$
\bu_{n+1} = \bu_n + \tau \bw_n,
$
where the damping parameter $\tau \in (0,1]$, which might depend on $n$, satisfies
\[
\| f- T_{\varepsilon}[\uve^n + \tau w_n] \|_{L^2(\Omega)}  <
\| f- T_{\varepsilon}[\uve^n] \|_{L^2(\Omega)}.
\]

We now explain the construction of $\bD\bT_{\varepsilon}[\bu_n]$. 
Evaluating $\sdd z_h(x_i;v_j)$, in view of \eqref{E:2Sc2Dif}, requires
knowing $z_h$ at $x_i^\pm=x_i \pm \rho \delta v_j$, which are not necessarily nodes
of $\mathcal N_h$. Since $x_i^\pm$ belong to two simplices of $\Th$,
and $z_h\in\Vh$, the values $z_h(x_i^\pm)$ can be determined in terms
of the barycentric coordinates of $x_i^\pm$. Therefore, if we define
$\sddb \bz(i;v_j) := \sdd z_h(x_i;v_j)$, then we realize that this operator
involves $2(d+1)+1$ components of $\bz$ and is thus sparse. 
We likewise define $\sddbp$ and $\sddbm$ to be 
the component-wise versions of $\sddp$ and $\sddm$.
The operator $\bT_\varepsilon$ reads
\begin{equation}\label{E:boldOp}
 \bT_\varepsilon[\bz](i)  :=
\min_{\bv=(v_j)_{j=1}^d\in\Vperpt} \left(\prod_{j=1}^d \sddbp \bz(i;v_j) -
\sum_{j=1}^d \sddbm \bz(i;v_j) \right) = f(x_i),
\end{equation}
according to \eqref{E:PracticalOp}.
\\
Let now $\bv^i=(v_j^i)_{j=1}^d\in\Vperpt$ be a set of directions that
realize the minimum of $\bT_\varepsilon[\bu_n](i)$ and denote 
$\bV:=(\bv^i)_{i=1}^N\in\mathbb{R}^{d\times d \times N}$, the collection
of the minimizing d-tuples $\bv^i$ for all $i = 1,\ldots N$. Combining
the above notations, we denote the matrix that contains the j-th 
minimizing directions for each node by $\bV_j \in \mathbb{R}^{d \times N}.$ 
This allows us to display our Jacobian in a vectorized form, using the notation 
$\sddb \bu_n(\bV_j) := (\sdd \bu_n(i;v_j^i))_{i=1}^N$, 
since \eqref{E:boldOp} gives for $\bz = \bu_n$:
$$
\bT_\varepsilon[\bu_n] = \bigodot_{j=1}^d \sddbp \bu_n(\bV_j) - \sum_{j=1}^d \sddbm \bu_n(\bV_j),
$$
where $\odot$ stands for the component-wise multiplication of vectors.
Using Danskin's Theorem \cite{BeNeOz} and the product rule, we can 
then obtain $ \bD \bT_{\varepsilon}[\bu_n]\bw_n.$
For that, we need to differentiate
 $\sddbp \bu_n(\bV_j).$ We observe that for each component,
$\sddbp \bu_n(i;v_j^i) = \max\{\sddb \bu_n(i;v_j^i),0\}$ is not
differentiable at $\sddb \bu_n(i;v_j^i)=0$. As a result, we use the so-called slant derivative
in the direction $\bw_n$ \cite{CNQ,HIK}, in order to compute:
\begin{align*}
\bD[ \sddbp \bu_n(\bV_j) ] \mathbf{w}_n &=
\bH^+[\sddb \bu_n(\bV_j)] \odot \sddb \bw_n(\bV_j),
\end{align*}
each component of which is equal to 
\begin{align*}
(\bD[ \sddbp \bu_n(\bV_j) ] \mathbf{w}_n)_i &=
\begin{cases}
\sddb \mathbf{w}_n(i;v_j^i) \quad & \textrm{if }\sddb \bu_n(i;v_j^i) > 0
\\
0 & \textrm{if } \sddb \bu_n(i;v_j^i) \leq 0.
\end{cases}
\end{align*}
Here $\bH^+$ is the operator that assigns 1
to a strictly positive component and zero otherwise.
Similarly, $\bD[ \sddbm \bu_n(\bV_j) ] \mathbf{w}_n =
\bH^-[\sddb \bu_n(\bV_j)] \odot \sddb \bw_n(\bV_j)$
where $\bH^-$ assigns $-1$ to a non-positive component and $0$ otherwise.
We are now ready to employ Danskin's Theorem \cite{BeNeOz} and the product
rule to obtain similarly to \cite{FrOb1}:
\small
\begin{equation*}\label{E:Jacobian}
  \bD \bT_{\varepsilon}[\bu_n]\bw_n \!=\! \sum_{j=1}^d \! \sddb \bw_n(\bV_j)
  \!\bigodot\! \Big( \bH^+[\sddb
  \bu_n(\bV_j)] \bigodot_{k \not = j}  \sddbp \bu_n(\bV_k)
- \bH^-[\sddb \bu_n(\bV_j)] \Big),
\end{equation*}
\normalsize
The presence of both operators $\bH^+$ and $\bH^-$ enforces
discrete convexity, and their definition at zero yields non-singular
Jacobians computationally.
This flexibility in choosing $\bH^+$ and
$\bH^-$ with vanishing argument is consistent with the definition of the slant
derivative for the max and min functions \cite{HIK}.

We initialize the Newton iteration with $\bu_0$ corresponding to the
Galerkin solution in $\Vh$ to the auxiliary problem
$\Delta u_0=(d! f)^{1/d}$ in $\Omega$ and $u_0=g$ on $\partial\Omega$,
as proposed in \cite{FrOb1}, but only for the coarser mesh
$h=2^{-5}$. For all subsequent refinements we interpolate the discrete
solution in the previous coarse mesh and use it as initial guess. This
greatly improves the residual error and leads to minimal or no damping.

\subsection{Accuracy}\label{S:Accuracy}

We examine the performance of our two-scale method mainly with
two examples, with smooth and discontinuous Hessians;
a third example entails an unbounded $f$.
For the first two examples we choose $\delta=h^{\alpha}$ and $\theta=h^{\beta}$
for appropriate $\alpha, \beta >0$ which yield provable rates of convergence
according to theory \cite{NNZ:17}. We stress that
smaller values of $\theta$ lead to similar convergence rates but affect the 
sparsity pattern of the matrix in the semi-smooth Newton iteration
because the number of search directions within $\St$ increase. We thus
choose $\theta$ consistent with theory \cite{NNZ:17}.
The computation of $\rho$ in \eqref{E:2Sc2Dif} is exact,
because $\Omega_h$ is a square, although need not be in general.
We stop the Newton iterations when 
	\begin{equation*}
	 \| f- T_{\varepsilon}[\uve^{n+1} ] \|_{L^2(\Omega)}  < 10^{-8}
	 \| f- T_{\varepsilon}[\uve^0] \|_{L^2(\Omega)}.
	\end{equation*}

\smallskip\noindent
{\bf Smooth Hessian:} We choose the solution $u$ and forcing $f$ to be
\[
u(x) = e^{|x|^2/2},
\quad
f(x) = (1+|x|^2) e^{|x|^2}
\quad\forall x\in\Omega.
\]
We choose $\delta, \theta \approx h^{1/2}$ on the basis of
\cite[Theorem 5.3]{NNZ:17},
and report the results in Table \ref{Ta:SmoothandC11}(a) and Figure \ref{F:SmRates}(a).
We observe linear experimental convergence rates with respect to $h$, thus better
than predicted in \cite{NNZ:17}.
The number $P=4(D-1)$ stands for the number of points $x_i\pm\delta
v_j$ used in the evaluation of the operator $T_\varepsilon$ at each
interior node $x_i\in\Nhi$ and for $D$ directions $v_j$ in a quarter
circle dictated by $\theta$.

\FloatBarrier
\begin{figure}[!htb]
	\includegraphics[scale=0.31]{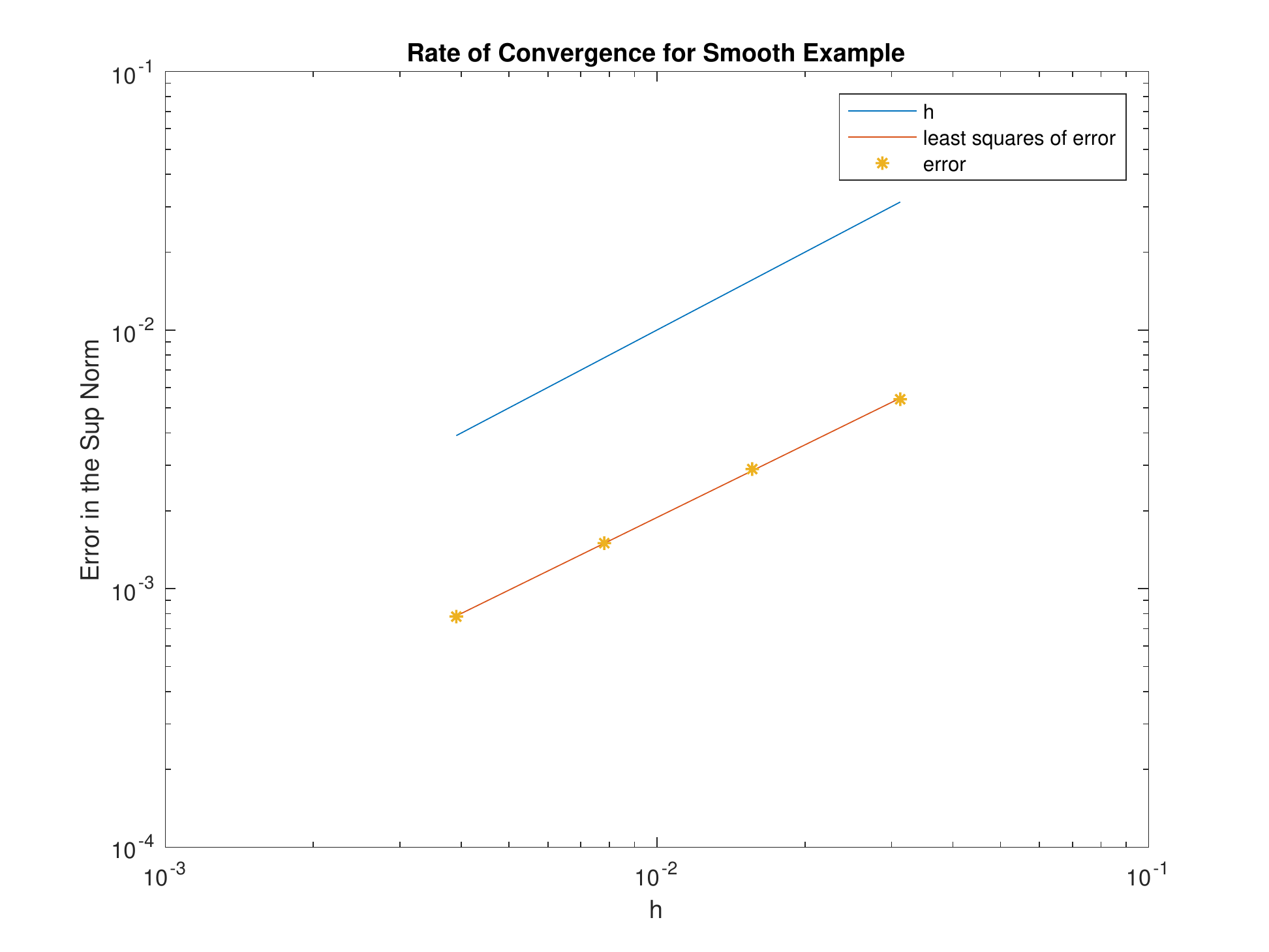}
	\includegraphics[scale=0.31]{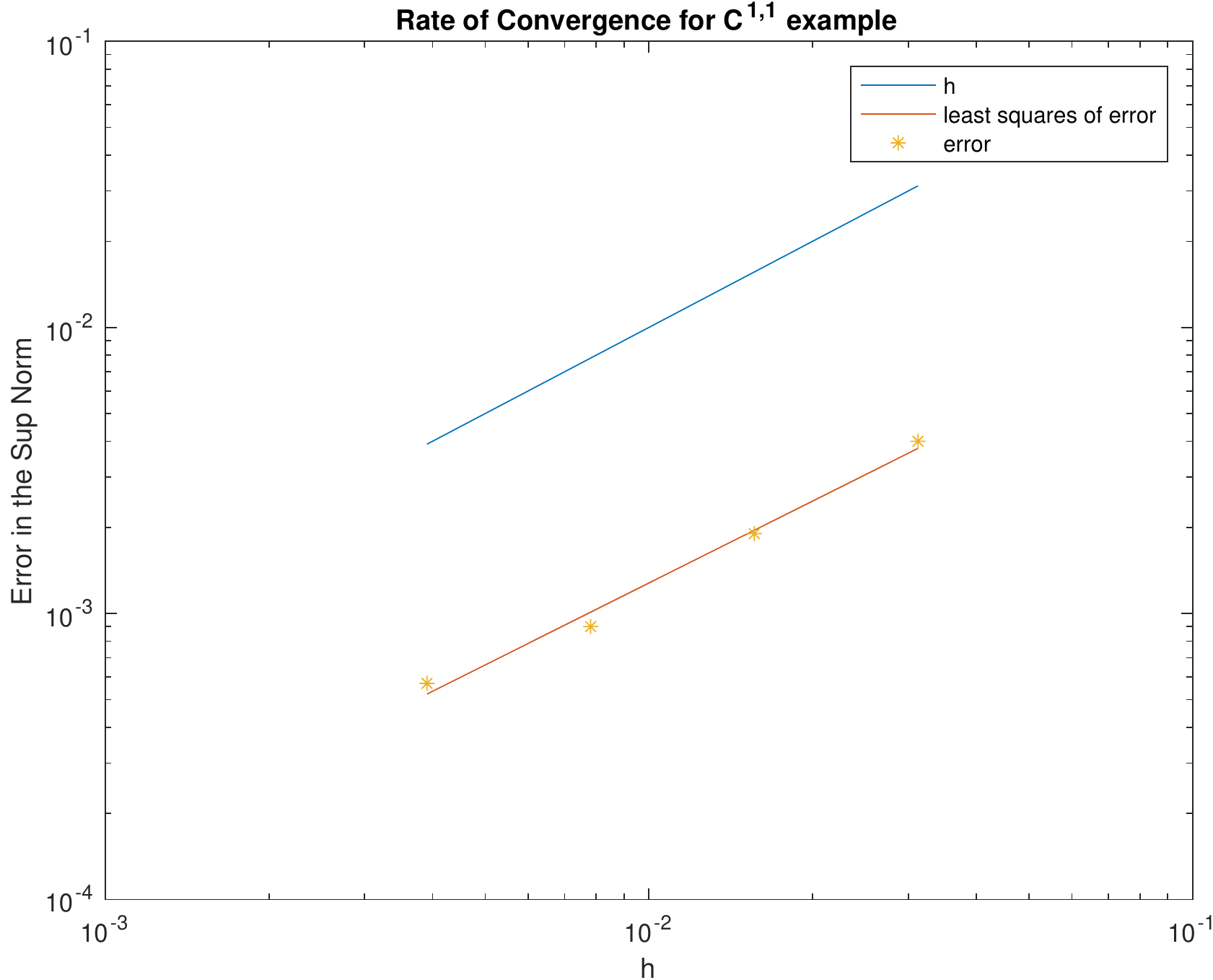}
	\caption{\small Experimental rates of convergence: the order is about
        $1$ in terms of $h$ for both the smooth Hessian with $\delta, \theta \approx h^{1/2}$ (left) and discontinuous Hessian $\delta \approx h^{4/5}, \theta \approx h^{2/5}$ (right).
	}
	\label{F:SmRates}
\end{figure}
\FloatBarrier

\FloatBarrier
\begin{table}[!htb]
\begin{center}
\begin{tabular}[t]{ | l | c | c | c |}
\hline
Degrees of freedom    & $P$: number of points & $L_{\infty}-$error  & Newton steps \\
\hline\hline
N= 1089, $h=2^{-5}$	 &	  16         	&  $5.4 \ 10^{-3}$			& 8  \\ 
\hline 
N=4225,  $h=2^{-6}$	  &          	24	 & 	$2.8 \ 10^{-3}$			&  7 \\
\hline
N=16641,  $h= 2^{-7}$	&	36	           & 	$1.5 \ 10^{-3}$			&  7 \\
\hline
N= 66049,  $h=2^{-8} $  	&	52	           & 	$7.8 \ 10^{-4}$			&   8  \\
\hline
\end{tabular}
\end{center}
\vskip0.05cm
\begin{center}
\begin{tabular}[t]{ | l | c | c | c |}
\hline
Degrees of freedom      & $P$: number of points & $L_{\infty}-$error  & Newton steps \\
\hline\hline
N= 1089, $h=2^{-5}$	    &	     20     &  $4.0 \ 10^{-3}$			& 10  \\ 
\hline
N=4225,  $h=2^{-6}$	    &	     28      & 	$1.9 \ 10^{-3}$			&  9 \\
\hline
N=16641,  $h= 2^{-7}$  &		36	           & 	$ 9.0 \ 10^{-4}$			&  9 \\
\hline
N= 66049,  $h=2^{-8} $  &		48           & 	$5.7 \ 10^{-4}$			&   9  \\
\hline
\end{tabular}
\end{center}
\FloatBarrier
\vskip0.05cm
\caption{\small Smooth Hessian with $\delta, \theta \approx h^{1/2}$ (top),
Discontinuous Hessian with $\delta \approx h^{4/5}, \theta \approx h^{2/5}$ (bottom).
The convergence rate is about linear in $h$ for both cases
(see Figure \ref{F:SmRates}),
whereas the number of Newton
steps seem insensitive to the dimension $N$ of the nonlinear system.}
\label{Ta:SmoothandC11}
\end{table}

\medskip\noindent
{\bf Discontinuous Hessian:}
We choose the solution $u$ and forcing function $f$ to be
\[
u(x) =  \frac{1}{2} \left(\max({|x-x_0|-0.2,0)}\right)^2, 
\quad
f(x) =\max{\left(1-\frac{0.2}{|x-x_0|},0\right)}
\quad\forall x\in\Omega,
\]
where $x_0=(0.5,0.5)$. Since $f=0$ in the ball centered at $x_0$ of
radius $0.2$, this example is degenerate elliptic.
We choose $\delta = O(h^{4/5})$ and $\theta = O(h^{2/5})$
on the basis of \cite[Theorem 5.6]{NNZ:17},
and observe experimentally again a linear decay rate in $h$, which is better than predicted.
This time \cite{NNZ:17} suggests a larger $\theta$, but we choose a
smaller $\theta$ without compromising the sparsity pattern of the Newton matrix.
As illustrated on \Cref{Ta:SmoothandC11}, despite its degeneracy and
lack of global regularity, this example does not exhibit any
problematic behavior compared to the smooth case.

We next explore the behavior of the operator $T_\varepsilon$ in
terms of the sign of the truncation error $E_\varepsilon[\uve]:=f-T_\varepsilon[\uve]$.
In \Cref{F:TruncErrSigns} (left) we split the interior nodes $\Nhi$
into three sets, using the threshold $\eps \approx 10^{-16}$
close to the machine precision of MATLAB:
blue nodes $x_i$ ($34\%$ of $\Nhi$) correspond to $E_\varepsilon[\uve](x_i)<-\eps$;
yellow nodes $x_i$ ($34\%$ of $\Nhi$) correspond to $E_\varepsilon[\uve](x_i)>\eps$;
and magenta nodes $x_i$ ($32\%$ of $\Nhi$) correspond
to $|E_\varepsilon[\uve](x_i)|\le \eps$.
Moreover, \Cref{F:TruncErrSigns} (left) displays with dashed lines the circle
of discontinuity $|x-x_0|=0.2$ and the two circles that are $\delta-$away from it.
We point out that all points between the outer and inner circle are affected by the
singularity, but they are mostly magenta nodes.

\FloatBarrier
\begin{figure}[!htb]
\includegraphics[scale=0.18]{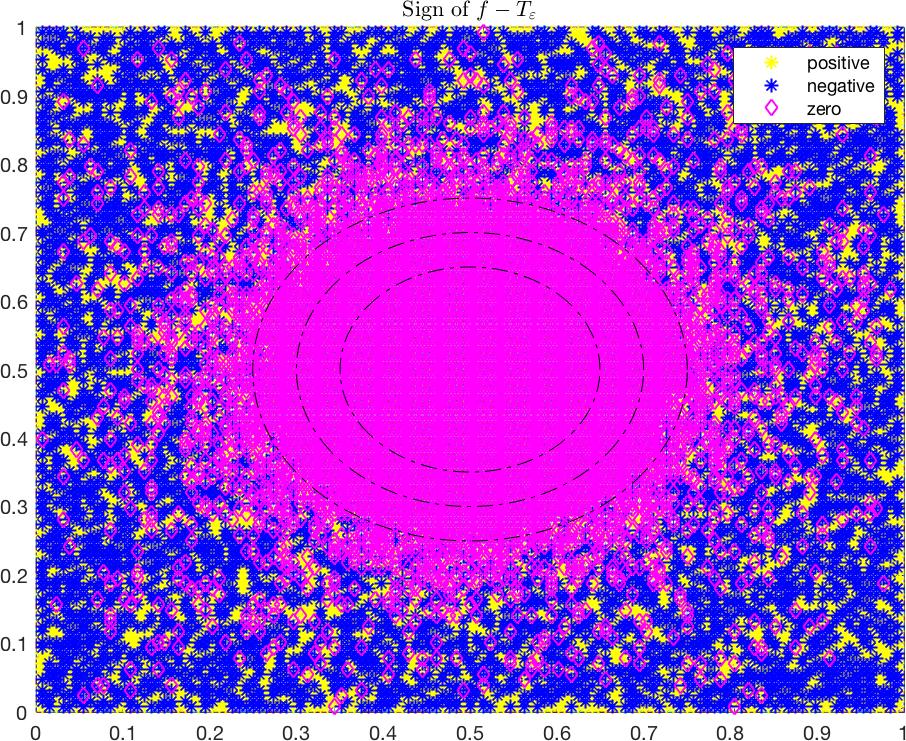}
\includegraphics[scale=0.18]{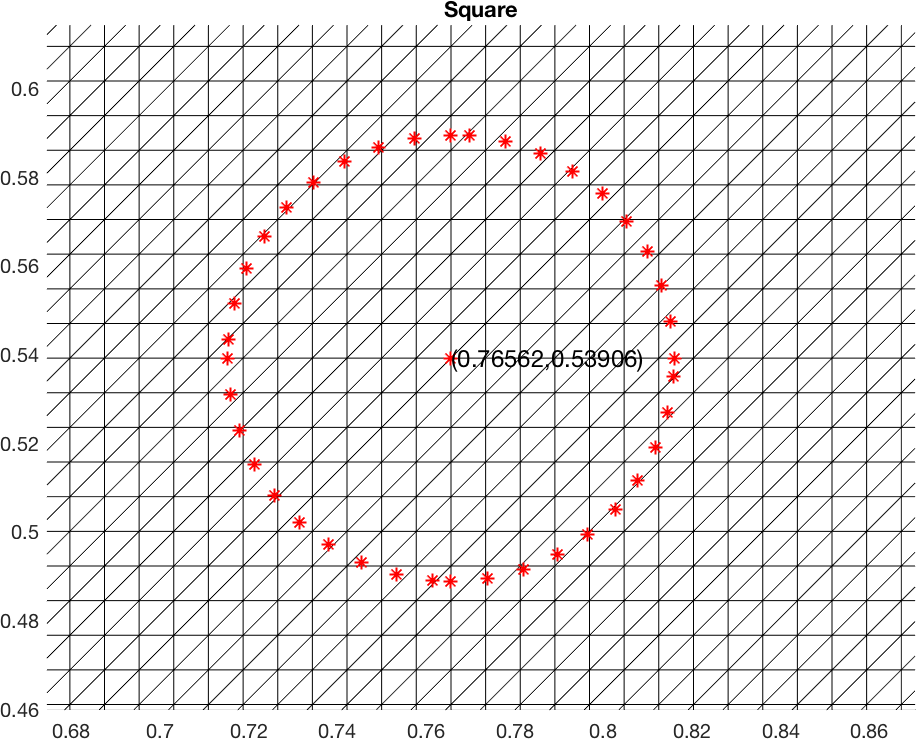}
\caption{\small
  (left) Sign of the truncation error $E_\varepsilon[\uve]=f-T_\varepsilon[\uve]$
  at nodes $x_i\in\Nhi$ for the example with discontinuous
  Hessian and $h=2^{-7}$: blue node $x_i$ if $E_\varepsilon[\uve](x_i)<-\eps$,
  yellow node $x_i$ if $E_\varepsilon[\uve](x_i)>\eps$, and magenta node $x_i$ if
  $|E_\varepsilon[\uve](x_i)| \le \eps$, where $\eps \approx 10^{-6}$.
  We observe that the region $|x_i-x_0| < 0.2$, where
  $f \equiv 0$, is magenta.
  (right) Set of directions in $\St$ centered at a node $x_i\in\Nhi$
  and scaled by $\delta$ for the same example;
  note that $\delta/h \approx 7$.}
	\label{F:TruncErrSigns}
\end{figure}
\FloatBarrier

Lastly, we use the same example to provide some insight on the
two-scale nature of our method. In \Cref{F:TruncErrSigns} (right) we display a
node $x_i = (0.7656,0.5391)$ within a zoomed mesh, and the thirty-six
directions $v_j$ in $\St$ scaled by $\delta$ which are used for the calculation
of $T_\varepsilon[\uve](x_i)$ for mesh size $h=2^{-7}$. We see that
for this specific instance, $\delta/h \approx 7$, and that most points
$x_i\pm\delta v_j$ are not nodes. We employ a fast search routine
within FELICITY to locate such points \cite{WalkerPaper,WalkerWeb}.

\medskip\noindent
{\bf Unbounded $f$:}
We finally present computational results for an example that
 does not fall within our theory because the right hand
 side $f$ is not uniformly bounded. More precisely, we
 consider the following $f$, which becomes unbounded near the
 corner $(1,1)$ of $\Omega$, and the corresponding exact solution
 $u$, which is twice differentiable
 in $\Omega$ but possesses an unbounded gradient near
 $(1,1)$ \cite{FrOb1}:
 \[
 u(x) =  -\sqrt{2-|x|^2},
 \quad
 f(x) = 2(2-|x|^2)^{-2}
 \quad\forall x\in\Omega.
 \]
\Cref{T:unbounded-example} shows that our method converges
as the meshsize $h$ decreases, but with a reduced rate and at the cost
of an increased number of Newton iterations.
We choose $\delta$ and $\theta$ similarly to the smooth Hessian case,
but without any theoretical justification from \cite{NNZ:17}.
We note that now we do not follow the approach of interpolating the coarse
solution to the finer mesh, because $u \notin\Wti(\Omega)$.
Instead, we use the initial guess that corresponds to
$\Delta u_0=(d! f)^{1/d}$, which introduces more damping, say
$\tau<1$, in the Newton method.
\FloatBarrier
\begin{table}[!htb]
	\begin{center}
		\begin{tabular}[t]{ | l | c | c | c |}
			\hline
			Degrees of freedom    & $P$: number of points & $L_{\infty}-$error  & Newton steps \\
			\hline\hline
			N= 1089, $h=2^{-5}$	 &	  16         	&  $8.3 \ 10^{-3}$			& 8  \\ 
			\hline 
			N=4225,  $h=2^{-6}$	  &     24	 		& 	$5.0 \ 10^{-3}$			&  15 \\
			\hline
			N=16641,  $h= 2^{-7}$	&	36	           & 	$3.3 \ 10^{-3}$			&  18 \\
			\hline
			N= 66049,  $h=2^{-8} $  	&	52	           & 	$2.0 \ 10^{-3}$			&  50   \\
			\hline
		\end{tabular}
	\end{center}
\caption{\small Unbounded $f$. We observe that the method converges,
but with a rate slower than linear and at the cost of increasing number of Newton iterations with each refinement.}
\label{T:unbounded-example}
\end{table}
\FloatBarrier

\noindent
{\bf Computational Performance:}
The process of locating
the triangle of the mesh containing $x_i\pm\delta v_j$ and computing the
barycentric coordinates is a rather small percentage of the total
computing time. For instance, for $h=2^{-7}$ and the smooth
Hessian, this represents just 3\% ($<3$ sec) of the total computation
time (90 sec). Because of the reduced sparsity pattern of the Newton
matrices, the most time demanding task of the method is solving
the linear systems, for which we use Matlab's backslash operator. This
takes 42.7\% of the total time. This computation is performed on an
Intel 2.2 GHz i7 CPU, 16 GB RAM using Matlab R2017b.

\medskip\noindent
{\bf Acknowledgments.}
We are indebted to S.W. Walker for providing assistance and guidance
with the software FELICITY and to H. Antil for numerous discussions
about implementing the 2-scale method and the semi-smooth
Newton solver.

\bibliographystyle{amsplain}

\end{document}